\documentclass{svjour3}
\smartqed  
\usepackage{graphicx}
\usepackage{latexsym,bm,amsmath,amssymb}
\usepackage{rotating}
\usepackage{multirow,booktabs,cases}
\usepackage[misc]{ifsym}
\usepackage[style=1]{mdframed}
\usepackage{algorithm,algorithmic}
\usepackage[colorlinks=true]{hyperref}
\hypersetup{urlcolor=blue, citecolor=blue,linkcolor=blue}
\usepackage{epstopdf}
\usepackage{bm}
\usepackage{fix-cm}

\usepackage[normalem]{ulem}

\DeclareMathOperator*{\argmin}{argmin}

\newcommand{\R}{\mathbb{R}}
\def\x{{\bf x}}
\def\y{{\bf y}}
\def\z{{\bf z}}

\def\u{{\bf u}}
\def\v{{\bf v}}
\def\t{{\bf t}}
\def\0{{\bf 0}}
\def\A{{\cal A}}
\def\B{{\cal B}}

\def\T{{\cal T}}
\def\G{{\cal G}}

\def\s{{\tilde{s}}}

\allowdisplaybreaks

\begin{document}
\graphicspath{{./PIC/}}

\title{Tensor-based Dinkelbach method for computing generalized tensor eigenvalues and its applications}
\titlerunning{TD method for minimal generalized eigenvalue and applications}     

\author{Haibin Chen \and Wenqi Zhu  \and Coralia Cartis}


\institute{
H. Chen \at School of Management Science, Qufu Normal University, Rizhao, Shandong, China, 276800. \\
\email{chenhaibin508@qfnu.edu.cn}
\and W. Zhu \at Mathematical Institute, Woodstock Road, University of Oxford, Oxford, UK, OX2 6GG. \email{wenqi.zhu@maths.ox.ac.uk}
\and C. Cartis \at Mathematical Institute, Woodstock Road, University of Oxford, Oxford, UK, OX2 6GG.
\email{cartis@maths.ox.ac.uk}}

\date{Received: date / Accepted: date}

\maketitle

\begin{abstract}
In this paper, we propose a novel tensor-based Dinkelbach--Type method for computing extremal tensor generalized eigenvalues. We show that the extremal tensor generalized eigenvalue can be reformulated as a critical subproblem of the classical Dinkelbach--Type method, which can subsequently be expressed as a multilinear optimization problem (MOP). The MOP is solved under a spherical constraint using an efficient proximal alternative minimization method which we rigorously establish the global convergence. Additionally, the equivalent MOP is reformulated as an unconstrained optimization problem, allowing for the analysis of the Kurdyka-{\L}ojasiewicz (KL) exponent and providing an explicit expression for the convergence rate of the proposed algorithm. Preliminary numerical experiments on solving extremal tensor generalized eigenvalues and minimizing high-order trust-region subproblems are provided, validating the efficacy and practical utility of the proposed method.

\keywords{Tensor eigenvalues - Generalized eigenpair - Fractional programming - Symmetric tensor - Global convergence.}
\subclass{65H17 - 15A18 - 90C3}
\end{abstract}

\section{Introduction}\label{Sec1}

With the growing research interest in efficient methods for processing large-scale and high-dimensional datasets, researchers have been actively developing algorithms for analyzing high-order tensors, which are multi-dimensional arrays of data. As a fundamental component of tensor analysis, tensor eigenvalues, and eigenvectors, along with their potential applications and computational methods, have been extensively studied in the literature, for example, in \cite{Cardoso99,Hillar2013,Nie2018}.
Several types of tensor eigenvalues have been introduced through different generalizations of the matrix case, each with specific applications. These include Z-eigenvalues \cite{Lim2005,Qi2005}, H-eigenvalues \cite{CPT2009,Qi2005}, D-eigenvalues, and others \cite{Qi08}.

The concept of tensor generalized eigenvalues was defined by Chang, Pearson, and Zhang \cite{CPT2009}. For given $m$th-order, $n$-dimensional real-valued symmetric tensors $\A$ and $\B$, the objective is to find $\lambda \in \R$ and $\x \in \R^n \setminus \{\0\}$ such that
\begin{eqnarray}
    \A \x^{m-1} = \lambda \B \x^{m-1}.
    \label{GEP}
\end{eqnarray}
A key advantage of the generalized eigenpair framework is its flexibility: by appropriately choosing $\B$, it can encompass multiple definitions of tensor eigenvalues, such as Z-eigenvalues, H-eigenvalues, and D-eigenvalues. For further details on selecting $\B$, we have constructed a summary table in Appendix \ref{Appendix: Def of Tensor eig}. We also refer readers to \cite{Kolda2014} for additional insights.
Computing eigenvalues of high-order tensors is an NP-hard problem, even for tensors with special structures \cite{Cui2014,Ng2009}. Consequently, researchers have primarily focused on computing their extremal eigenvalues. To the best of our knowledge, various numerical algorithms have been developed to compute extremal Z-eigenvalues, H-eigenvalues, or D-eigenvalues. These include the high-order power method \cite{Ng2009}, the SOS polynomial algorithm \cite{Cui2014}, and the Newton method, among others \cite{CQ2015}. However, when \(m\) is even and \(\B\) is positive definite, the well-known numerical method for computing generalized tensor eigenvalues is the generalized eigenproblem adaptive power (GEAP) method, proposed by Kolda et al. in \cite{Kolda2014}. After that, an adaptive gradient (AG) approach was presented by Yu et al. in 2016 to establish both global convergence and linear
convergence rates under certain appropriate conditions \cite{Yu2016}. Moreover, for the generalized eigenvalue of tensors with certain special structures, Zhao et al. proposed two convergent gradient projection techniques to address weakly symmetric tensors \cite{ZYL2017}. Very recently, 
a normalized Newton method named NNGEM was presented to compute the generalized tensor eigenvalue problem in \cite{PAH2024}. It was proved that 
the NNGEM approach has a higher rate of convergence than the AP and GEAM approaches. Note that, the algorithms above mainly focus on solving as many eigenpairs as possible, rather than the extreme eigenpairs.

In this paper, we reformulate the extremal generalized tensor eigenvalue problem in \eqref{GEP} as a subproblem of the classical Dinkelbach--Type method.
We outline this reformulation through the following analysis. By taking the dot product with \(\x\) on both sides of \eqref{GEP}, any solution to \eqref{GEP} satisfies
\[
\lambda = \frac{\A\x^m}{\B\x^m}.
\]
The extremal generalized eigenpairs can thus be computed by solving the following nonlinear program \cite{Kolda2014}
\[
\min_{\x \in S} \frac{\A\x^m}{\B\x^m}, \quad \text{and} \quad \max_{\x \in S} \frac{\A\x^m}{\B\x^m},
\]
where \(S = \{\x : \|\x\|^m = 1\}\). These relates to a class of single-ratio fractional programming problems in the form
\begin{equation}\label{e1}
\bar{\theta} = \min_{\x \in S} \frac{f(\x)}{g(\x)},
\end{equation}
and can be addressed using the classical Dinkelbach--Type method \cite{Dinkelbach1967}.

Using this tensor representation of homogeneous polynomials in \eqref{e1}, we propose a tensor-based Dinkelbach--Type method for computing the minimal generalized eigenvalue of two symmetric tensors. The proposed algorithm involves finding the root of the equation \(F(\bar{\theta}) = 0\), where \(F(\bar{\theta})\) is the optimal value of the following parametric program
\begin{equation}\label{e2}
F(\bar{\theta}) = \min_{\x \in S} \left(f(\x) - \bar{\theta} g(\x)\right).
\end{equation}
which shares the same global optimal solutions as \eqref{e1}.
Note that are homogeneous polynomials \(f(\x)\) and \(g(\x)\)  of even degree \(d\).  It is well known that homogeneous polynomials can be reformulated as high-order symmetric tensors applied to a vector of variables. Therefore, we show that \eqref{e2} can be solved by a proximal alternative minimization algorithm through a multilinear optimization framework. The PAM algorithm alternates between updating each variable (or block of variables) while keeping the others fixed, incorporating a proximal step to ensure both convergence and stability. The subproblems in PAM are of lower dimensions, and each subproblem has a closed-form solution.


The main contributions of this paper are twofold:

\begin{enumerate}
    \item First, we establish that (\ref{e2}) is equivalent to a multilinear optimization problem (MOP) under the condition that \( F(\bar{\theta}) \) is concave. For cases where the objective function is not concave, we introduce an augmented model with a regularization term to ensure the concavity of the resulting model. Furthermore, we prove that the augmented model and (\ref{e2}) share the same optimal solutions under unit spherical constraints.

    \item Second, leveraging the multilinear model corresponding to the augmented concave model, we propose a proximal alternating minimization (PAM) algorithm. Notably, the PAM algorithm can be interpreted as a block coordinate descent (BCD) method (see \cite{Tseng2001}) with a cyclic update rule and proximal terms tailored for polynomial optimization models. When applied to (\ref{e2}) under unit spherical constraints, the subproblems of the PAM algorithm have analytic optimal solutions.
    We rigorously establish the global convergence of the PAM algorithm, using the analysis of the Kurdyka-{\L}ojasiewicz (KL) exponent, under mild assumptions, we provide an explicit expression for the convergence rate of the proposed algorithm. Numerical experiments are conducted to compute the minimal generalized eigenvalue of symmetric tensors. Additionally, with slight modifications, the PAM algorithm can also be applied to the minimization of the high-order trust-region subproblems on the boundary.
\end{enumerate}

The remainder of this paper is organized as follows. Section \ref{Sec2} provides a review of relevant symbols, definitions of tensors, and the Dinkelbach--Type method. In Section \ref{Sec3}, we present the proposed PAM algorithm and establish its convergence and convergence rate under mild assumptions. Section \ref{Sec4} presents preliminary numerical results, demonstrating the application of the proposed algorithm for computing the minimal generalized eigenvalue of symmetric tensors and minimizing high-order trust-region subproblems on the boundary.
Finally, Section \ref{Sec5} offers concluding remarks and outlines potential future research directions.

\section{Preliminaries}\label{Sec2}

In this section, we recall some notations and useful preliminaries.
For a positive integer $n$, let $[n]=\{1,2,\cdots,n\}$. Real scalars from $\R$ are denoted by unbolded lowercase letters such as $\lambda$ and $a$, while vectors are denoted by bolded lowercase letters such as $\x$ and $\y$. Let $\R^n$ ($\R^n_+$) represent the set of all $n$-dimensional vectors (nonnegative vectors), and $\R^{m\times n}$ represent the set of all $m\times n$ matrices. Matrices are denoted by uppercase letters such as $A$ and $B$, while tensors are denoted by uppercase bold letters such as $\A$ and $\B$. $
\mathbb{R}_{p}[\s],  \text{where } p \geq 1 \text{ and } \s = [s_1, \dotsc, s_n]^\top,
$
represents the space of polynomials in \(n\) variables with real coefficients and a maximum degree of \(p\).

An order $m$ dimension $n$ tensor $\A$ is defined as $\A=(a_{i_1i_2\cdots i_m})$, where $i_1,\cdots,i_m \in [n]$ and $a_{i_1i_2\cdots i_m}$ is the $(i_1,\cdots,i_m)$th entry of $\A$. A tensor $\A$ is called symmetric if
\[
a_{i_1i_2\cdots i_m}=a_{i_{\sigma(1)}i_{\sigma(2)}\cdots i_{\sigma(m)}}
\]
for any permutation $\sigma$ of $[m]$. We denote the set of all symmetric tensors of order $m$ and dimension $n$ by $S^{m,n}$.

Let $\A=(a_{i_1i_2\cdots i_m})$ and $\B=(b_{i_1i_2\cdots i_m})\in S^{m,n}$ be two symmetric tensors. The inner product of $\A$ and $\B$ is defined as:
\[
\langle\A,~ \B\rangle=\sum_{i_1,i_2,\cdots,i_m\in[n]}a_{i_1i_2\cdots i_m}b_{i_1i_2\cdots i_m}.
\]
Let $\R_m[\x]$ denote the set of real polynomials of degree at most $m$. For $f(\x)\in \R_m[\x]$ and $\x\in\R^n$, its corresponding tensor is denoted by $\A_f=(a_{i_1i_2\cdots i_m})\in S^{m, n}$:
\[
f(\x)=\A_f\x^m=\langle\A_f,~ \underbrace{\x\circ\x\circ\cdots\circ\x}_m\rangle=\sum_{i_1,\cdots,i_m\in[n]}a_{i_1i_2\cdots i_m}x_{i_1}x_{i_2}\cdots x_{i_m},
\]
where ``$\circ$" denotes the outer product, and $\x^m=\underbrace{\x\circ\x\circ\cdots\circ\x}_m$ denotes a symmetric rank-1 tensor.
A tensor $\A\in S^{m,n}$ is called a positive semidefinite tensor if $m$ is even and $f(\x)=\A\x^m\geq 0$ for all $\x$. If the strict inequality $f(\x)>0$ holds for all $\x\neq\mathbf{0}$, then $\A$ is called a positive tensor.

\section{PAM for Tensor Generalized Eigenvalues Problems}\label{Sec3}

In this section, we present the main results of the paper. After revisiting the classic Dinkelbach--Type method, we introduce a proximal alternative minimization (PAM) method to compute the optimal solution of (\ref{e1}). A key step in this process is demonstrating the equivalence between (\ref{e2}) and a multilinear programming problem.

\subsection{Dinkelbach-Type Method and the Relationship between \eqref{e1} and \eqref{e2}}

It has been established in the literature that \eqref{e1} and \eqref{e2} share the same global optimal solution \cite{Dinkelbach1967}. This result is stated in Proposition \ref{prop1}. Consequently, minimizing \eqref{e1} can be achieved through an iterative minimization algorithm for \eqref{e2}, commonly referred to as the Dinkelbach--Type method, which is presented in Algorithm \ref{alg31}.  However, solving the subproblem in Algorithm \ref{alg31} is generally challenging even for some special fractional programming problems. Recently, Zhang and Li proposed a proximity-gradient subgradient algorithm to address a category of nonconvex fractional optimization problems \cite{ZL2022}, where $g(\x)$ is convex
but may not be smooth and $f(\x)$ is potentially nonconvex and nonsmooth. Furthermore, more general cases for fractional programming are considered in \cite{Bot17,Bot22,Bot23}.
Fortunately, in our specific context of finding the generalized tensor eigenvalue, the eigenvalue formulation ensures that both \(f(\x)\) and \(g(\x)\) are homogeneous. Leveraging this property, we prove a stronger result: \eqref{e1} and \eqref{e2} not only share the same global solution but also have the same KKT points. Furthermore, we provide a homogeneous tensor formulation for both \eqref{e1} and \eqref{e2}.
These formulations play a critical role in transforming the problem into a multilinear optimization framework, enabling us to efficiently solve the subproblem in Algorithm \ref{alg31}.

\begin{proposition}\label{prop1}{\rm\cite{Dinkelbach1967}} Assume that $S\subseteq\R^n$ is compact, and $f(\x), g(\x)$ are homogeneous polynomials such that $g(\x)\geq 0, \x\in S$. Then for problems (\ref{e1}) and (\ref{e2}), it always holds that
$F(\bar{\theta})=0$. Furthermore, $\x^*$ is an optimal solution of (\ref{e1}) if and only if $\x^*$ is an optimal solution of $F(\bar{\theta})$ in (\ref{e2}), and if $\theta=\hat{\theta}$ in (\ref{e2}) satisfying $F(\hat{\theta})=0$, then $\hat{\theta}=\bar{\theta}$.
\end{proposition}

\begin{algorithm}[!htbp]
\caption{Dinkelbach-type Algorithm for \eqref{e1}.}\label{alg31}
\begin{algorithmic}[1]
\STATE Set a tolerance level, \texttt{TOL} and $ k_{\max}$. Choose $\x^0\in S$ and set $\theta_1=\frac{f(\x^0)}{g(\x^0)}$. Let $k=1$.\medskip

\STATE While  $|F(\theta_k)| \ge \texttt{TOL} $ or $k \le k_{\max}$

\STATE Find
\begin{eqnarray}
    \x^k=\arg\min_{\x\in S}\left(f(\x)-\theta_k g(\x)\right)
    \label{Key subprob of Dink}
\end{eqnarray}(This can be solved efficiently by Algorithm \eqref{alg32}. \medskip
\STATE If $|F(\theta_k)| \le \texttt{TOL} $, stop. Then $\x^{k}$ is an optimal solution of (\ref{e1}) with optimal value $\theta_k$.\medskip

\STATE If $|F(\theta_k)| \ge \texttt{TOL}$, take $\theta_{k+1}=\frac{f(\x^k)}{g(\x^k)}$.\medskip

\STATE Set $k:=k+1$ and go to step 2.
\end{algorithmic}
\end{algorithm}







\begin{remark}
The core idea of the Dinkelbach--Type method (see Algorithm \ref{alg31}) is to construct a nonincreasing sequence \(\{\theta_k\}\), which generates a nondecreasing sequence of function values \(\{F(\theta_k)\}\), continuing until the stopping criterion \(F(\theta_k) = 0\) is satisfied. We remark that Algorithm \ref{alg31} is provably convergent \cite{CFS1985,Dinkelbach1967}, with the convergence result included in Appendix \ref{appendix convergence of algo 1} for completeness. It is also worth highlighting that the primary computational cost of this algorithm lies in the crucial step 2.
\end{remark}

Generally speaking, the subproblem is difficult to solve unless the constraint set $S$ and the objective function possesses special structures. Therefore, to efficiently compute the subproblem $\x^k = \arg\min_{\x \in S}\left(f(\x) - \theta_k g(\x)\right)$, the homogeneous structure of the objective function should be considered.

The homogeneous polynomial $f(\x)$ can always be represented as the inner product between its coefficient tensor $\A_f$ and a symmetric rank-1 tensor $\x \circ \x \circ \cdots \circ \x$. Without loss of generality, let $\A(\theta) = (a^{\theta}_{i_1 i_2 \ldots i_d})$, where $i_1, i_2, \ldots, i_d \in [n]$ is a $d$-th order, $n$-dimensional tensor such that
\[
f(\x) - \theta g(\x) = \A(\theta) \x^d = \sum_{i_1, i_2, \ldots, i_d = 1}^n a^{\theta}_{i_1 i_2 \ldots i_d} x_{i_1} x_{i_2} \ldots x_{i_d}.
\]
In the following analysis, we consider the problem (\ref{e1}) with a unit sphere constraint. The proposition below shows that, from an optimization perspective, the optimal solution of (\ref{e1}) must be a critical point of $F(\bar{\theta})$. Furthermore, a new relationship between KKT points \eqref{e1}-\eqref{e2} and Z-eigenvector is proved (The notion of Z-eigenvector is referred to in Appendix \ref{Appendix: Def of Tensor eig}).

\begin{proposition}\label{prop3} Let $\x^*$ be an optimal solution of (\ref{e1}) with $\bar{\theta}=\frac{f(\x^*)}{g(\x^*)}$, and $S:=\{\x\in\R^n~|~\|\x\|=1\}$, then it holds that $\nabla f(\x^*)-\bar{\theta}\nabla g(\x^*)=\0$, and $\x^*$ is a KKT point of (\ref{e1}) and (\ref{e2}). Furthermore, assume that $\bar{\x}\neq\x^*$ is an arbitrary KKT point of
(\ref{e1}), then $\bar{\x}$ is a KKT point of (\ref{e2}) if and only if $\bar{\x}$ is a Z-eigenvector of the corresponding tensor $\G$, where $g(\x)=\G\x^m$.
\end{proposition}
\proof.
By Proposition \ref{prop1}, we know that $\x^*$ is an optimal solution of $F(\bar{\theta})$.
Since the independent constraint qualification is satisfied, we know that $\x^*$ is a KKT point of (\ref{e1}) and (\ref{e2}) simultaneously.

Suppose that $\bar{\x}\in S$ is an arbitrary KKT point of (\ref{e1}) and (\ref{e2}), then there are multipliers $\bar{\lambda}$, $\bar{\mu}\in\mathbb{R}$
such that
\begin{equation}\label{e3}
\left\{
\begin{aligned}
&\frac{g(\bar{\x})\nabla f(\bar{\x})-f(\bar{\x})\nabla g(\bar{\x})}{g^2(\bar{\x})}-\bar{\lambda}\bar{\x}=\0,\\
&\nabla f(\bar{\x})-\bar{\theta}\nabla g(\bar{\x})-\bar{\mu}\bar{\x}=\0.
\end{aligned}
\right.
\end{equation}
By the first equation of (\ref{e3}) and a direction computation, it follows that
$$
\bar{\lambda}=0~\mbox{and}~\nabla f(\bar{\x})-\frac{f(\bar{\x})}{g(\bar{\x})}\nabla g(\bar{\x})=\0.
$$
Combining this with the second equation of (\ref{e3}) and $\nabla g(\bar{\x})=m\G\bar{\x}^{m-1}$, we obtain that
\begin{equation}\label{e31}
m\left(\frac{f(\bar{\x})}{g(\bar{\x})}-\bar{\theta} \right)\G\bar{\x}^{m-1}=\bar{\mu}\bar{\x}.
\end{equation}
Since $\bar{\x}\neq\x^*$, by Proposition \ref{prop1} again, we know that $\frac{f(\bar{\x})}{g(\bar{\x})}-\bar{\theta}\neq 0$ and
$$
\G\bar{\x}^{m-1}=\frac{\bar{\mu}}{m\left(\frac{f(\bar{\x})}{g(\bar{\x})}-\bar{\theta} \right)}\bar{\x},
$$
which implies that $\bar{\x}$ is a Z-eigenvector of $\G$.

For the sufficient direction, suppose that $\bar{\x}$ is a Z-eigenvector of $\G$,
then there is a scalar $a\in\mathbb{R}$ such that
$$
\nabla g(\bar{\x})=m\G\bar{\x}^{m-1}=a\bar{\x}.
$$
By the fact that $\bar{\x}$ is a KKT point of (\ref{e1}), and by the first equation of (\ref{e3}) again, we have that
\begin{equation}\label{e32}
\nabla f(\bar{\x})=a\frac{f(\bar{\x})}{g(\bar{\x})}\bar{\x}.
\end{equation}
Let $\bar{\mu}=a\left(\frac{f(\bar{\x})}{g(\bar{\x})}-\bar{\theta} \right)$. By (\ref{e32}), it follows that
$$
\nabla f(\bar{\x})-\bar{\theta}\nabla g(\bar{\x})-\bar{\mu}\bar{\x}=\0,
$$
and $\bar{\x}$ is a KKT point of (\ref{e2}), and the desired results hold.
\qed

\subsection{PAM Algorithm for Minimizing Dinkelbach--Type Subproblem \eqref{Key subprob of Dink}}

In this subsection, we demonstrate that the minimization of \eqref{Key subprob of Dink} can be reformulated as a multilinear optimization problem under unit spherical constraints. Consequently, the Proximal Alternating Minimization (PAM) algorithm can be efficiently applied to solve \eqref{Key subprob of Dink} under these constraints.
The PAM  algorithm minimizes alternately over each variable (or block of variables), treating the remaining variables as fixed, with each subproblem admitting a closed-form solution.

We first recall the following result to establish the equivalence between a homogeneous polynomial and a multilinear optimization problem (\cite[Corollary 4.2]{CSLZ2012}).

\begin{lemma}\label{lema1}
If $\A$ is a $d$-th order symmetric (or super-symmetric) tensor with $\hat{N}\in\R^{m\times n}$ being a matrix,
then it holds that
$$
\max_{\|\hat{N}\x\|=1}|\A\x^m|=\max_{\|\hat{N}\x^{(i)}\|=1,i=1,2,\ldots,d}|\langle\A,\x^{(1)}\circ\x^{(2)}\circ\cdots\circ\x^{(d)}\rangle|,
$$
where $\x,\x^{(1)},\x^{(2)},\ldots,\x^{(d)}\in\R^n$.
\end{lemma}

Let $\hat{N}$ be the identity matrix, using Lemma \ref{lema1},  we deduce the following result.
\begin{theorem}\label{theorem1}Suppose $S:=\{\x\in\R^n~|~\|\x\|=1\}$ and $g(\x)>0, \x\in S$. If $f(\x)-\theta g(\x)=\A(\theta)\x^d$ is concave for any $\x\in \mathbb{R}^n$, then it holds that
$$
\min_{\x\in S}\A(\theta)\x^d=\min_{\x,\y,\ldots,\z\in S}\langle\A(\theta),\x\circ\y\circ\cdots\circ\z\rangle.
$$
\end{theorem}
\proof By Lemma \ref{lema1}, denote $\hat{N}$ as an identity matrix with proper dimensions, it holds that
\begin{equation}\label{e4}
\max_{\x\in S}|\A(\theta)\x^d|=\max_{\x,\y,\ldots,\z\in S}|\langle\A(\theta),\x\circ\y\circ\cdots\circ\z\rangle|.
\end{equation}
By conditions that $\A(\theta)\x^d$ is a concave function, it follows that
$$
\y^\top\A(\theta)\x^{m-2}\y=\A(\theta)\y^2\x^{m-2}\leq 0,~\forall~\y\in\R^n,
$$
which implies that $\A(\theta)\x^d\leq 0$ (let $\y=\x$ above) and
\begin{equation}\label{e5}
\max_{\x\in S}|\A(\theta)\x^d|=\max_{\x\in S}(-\A(\theta)\x^d).
\end{equation}
On the other hand, by the multilinear property of $\langle\A(\theta),\x\circ\y\circ\cdots\circ\z\rangle$, we have that
\begin{equation}\label{e6}
\begin{aligned}
\max_{\x,\y,\ldots,\z\in S}|\langle\A(\theta),\x\circ\y\circ\cdots\circ\z\rangle|&=\max_{\x,\y,\ldots,\z\in S}\langle\A(\theta),\x\circ\y\circ\cdots\circ\z\rangle\\
&=\max_{\x,\y,\ldots,\z\in S}(-\langle\A(\theta),\x\circ\y\circ\cdots\circ\z\rangle).
\end{aligned}
\end{equation}
Combining \eqref{e4}--\eqref{e6}, we know that
$$
\min_{\x\in S}\A(\theta)\x^d=\min_{\x,\y,\ldots,\z\in S}\langle\A(\theta),\x\circ\y\circ\cdots\circ\z\rangle,
$$
and the desired result follows.
\qed

Theorem \ref{theorem1} establishes that the optimal value of a homogeneous function defined on a sphere equals the optimal value of a multilinear function with the same coefficient tensors. This result provides a crucial guarantee for solving the subproblem in the Dinkelbach--Type algorithm.

Generally speaking, for Lemma \ref{lema1}, the absolute value sign in the objective function on the right-hand side can be removed due to the symmetry of the constraint set. However, it is important that without the concavity condition, Lemma \ref{lema1}  does not necessarily lead to the result in   Theorem \ref{theorem1}. A counterexample is provided below to illustrate this point.


\begin{example}
 \label{counter examples}
Denote two diagonal matrices such that
$$
A_1=
\left(\begin{matrix}
1 &~ 0\\
0 &~ -2
\end{matrix}\right),~~~
A_2=
\left(\begin{matrix}
2 &~ 0\\
0 &~ 4
\end{matrix}\right).
$$
The corresponding homogeneous polynomial and multilinear polynomial are given below:
$$
f_{A_1}(\x)=x_1^2-2x_2^2,~~~~f_{A_1}(\x,\y)=x_1y_1-2x_2y_2,
$$
and
$$
f_{A_2}(\x)=2x_1^2+4x_2^2,~~~~f_{A_2}(\x,\y)=2x_1y_1+4x_2y_2.
$$
It is easy to know that $f_{A_1}(\x)$ and $f_{A_2}(\x)$ are not concave functions. By a direct computation, we have that
$$
\min_{\x\in S} f_{A_1}(\x)=\min_{\x,\y\in S} f_{A_1}(\x,\y)=-2,
$$
and
$$
\min_{\x\in S} f_{A_2}(\x)=2,~~~~~\min_{\x,\y\in S} f_{A_2}(\x,\y)=-4,
$$
where $S=\{\x\in\R^2~|~x_1^2+x_2^2=1\}$.
\end{example}

Since \(f(\x) - \theta g(\x)\) in \eqref{Key subprob of Dink} is generally not a concave function for \(\x \in S\), it is necessary to introduce an augmented model with a regularization term to ensure its concavity in order to apply the equivalence relationship in Theorem \ref{theorem1}. A bound for the size of the regularization term is provided in Lemma \ref{lema2}.

\begin{lemma}\label{lema2} Suppose that $d$ is even and
$$h_{\alpha}(\x)=f(\x)-\theta g(\x)-\alpha\|\x\|^d=\A(\theta)\x^d-\alpha\|\x\|^d, \x\in S.$$
Then $h_{\alpha}(\x)$ is a concave function for any $\x\in S$ if $\alpha\geq \|\A(\theta)\|_F$.
\end{lemma}
\proof
First of all, we know that
$$
h_{\alpha}(\x)=\A(\theta)\x^d-\alpha\|\x\|^d=\A(\theta)\x^d-\alpha(\x^\top\x)^{\frac{d}{2}}.
$$
Then it holds that
$$
\nabla h_{\alpha}(\x)=d\A(\theta)\x^{d-1}-\alpha d(\x^\top\x)^{\frac{d}{2}-1}\x,
$$
$$
\nabla^2 h_{\alpha}(\x)=d(d-1)\A(\theta)\x^{d-2}-\alpha d(d-2)(\x^\top\x)^{\frac{d}{2}-2}\x\x^\top-\alpha d(\x^\top\x)^{\frac{d}{2}-1}I.
$$
By conditions, for any $\y\in\R^n$, we obtain that
$$
\begin{aligned}
\y^\top\nabla^2 h_{\alpha}(\x)\y=&d(d-1)\A(\theta)\y^2\x^{d-2}-\alpha d(d-2)\|\x\|^{d-4}(\x^\top\y)^2-\alpha d\|\x\|^{d-2}(\y^\top\y)^2\\
=&d(d-1)\A(\theta)\y^2\x^{d-2}-\alpha d(d-2)r^{d-4}(\x^\top\y)^2-\alpha dr^{d-2}(\y^\top\y)^2\\
\leq & d(d-1)\|\A(\theta)\|_F\|\y\|^2r^{d-2}-\alpha d(d-2)r^{d-2}\|\y\|^2-\alpha dr^{d-2}\|\y\|^2\\
=&\left[d(d-1)\|\A(\theta)\|_F-\alpha d(d-2)-\alpha d\right]r^{d-2}\|\y\|^2\\
=&\left[d(d-1)\|\A(\theta)\|_F-\alpha d(d-1)\right]r^{d-2}\|\y\|^2,
\end{aligned}
$$
which implies that $\y^\top\nabla^2 h_{\alpha}(\x)\y\leq 0$ when $\alpha\geq \|\A(\theta)\|_F$, and the desired result holds.
\qed

Combining Theorem \ref{theorem1} and Lemma \ref{lema2}, the following result is true.

\begin{corollary}\label{corollary2} Let $S$ be defined as in Theorem \ref{theorem1}.
Suppose $h_{\alpha}(\x)$ is defined as in  Lemma \ref{lema2} with $\alpha\geq \|\A(\theta)\|_F$, then it holds that
$$
\arg\min_{\x\in S}(f(\x)-\theta g(\x))=\arg\min_{\x\in S} h_{\alpha}(\x) = \min_{\x,\y,\cdots,\u,\z\in S}h_{\alpha}(\x,\y,\cdots,\u,\z).
$$
\end{corollary}
\begin{proof}
    Clearly, combining Theorem \ref{theorem1} and Lemma \ref{lema2} gives
$$
\arg\min_{\x\in S} h_{\alpha}(\x) = \min_{\x,\y,\cdots,\u,\z\in S}h_{\alpha}(\x,\y,\cdots,\u,\z).
$$
Moreover, since  $S:=\{\x\in\R^n~|~\|\x\|=1\}$, $
\arg\min_{\x\in S}(f(\x)-\theta g(\x)) = \arg\min_{\x\in S}(f(\x)-\theta g(\x)  -\alpha\|\x\|^d)
$ for any $\alpha >0$.
\end{proof}

Note that the minimization of \(h_{\alpha}(\x,\y,\cdots,\u,\z)\) is a multilinear programming problem, written as
\begin{equation}\label{e7}
\min_{\x,\y,\cdots,\u,\z\in S} h_{\alpha}(\x,\y,\cdots,\u,\z) = \min_{\x,\y,\cdots,\u,\z\in S} \langle\A(\theta), \x\circ\y\circ\cdots\circ\z\rangle - \alpha\langle\x,\y\rangle\cdots\langle\u,\z\rangle.
\end{equation}
When \(\x = \y = \cdots = \u = \z\), we denote \(h_{\alpha}(\x) = h_{\alpha}(\x,\y,\cdots,\u,\z)\) for simplicity.

We are now prepared to introduce the PAM framework tailored for solving the subproblem \eqref{Key subprob of Dink}. The details are provided in Algorithm \ref{alg32}.

\begin{algorithm}[!htbp]
\caption{Proximal Alternating Minimization (PAM) Algorithm.}\label{alg32}
\begin{algorithmic}[1]
\STATE Input $\A(\theta)$, $\alpha\geq\|\A\|_F$, $\x^{(0)},\y^{(0)},\cdots,\u^{(0)},\z^{(0)}\in S$, $\epsilon>0$,$\gamma_i\geq0, i=1,2,\ldots,d$.\medskip

\STATE For $k=0,1,2,\ldots,N$ do. \medskip

\STATE Update $\x^{(k+1)},\y^{(k+1)},\ldots,\u^{(k+1)},\z^{(k+1)}$ sequentially via
$$
\left\{
\begin{aligned}
&\x^{(k+1)}=\arg\min_{\x\in S}h_{\alpha}(\x,\y^{(k)},\cdots,\u^{(k)},\z^{(k)})+\frac{\gamma_1}{2}\|\x-\x^{(k)}\|^2,\\
&\y^{(k+1)}=\arg\min_{\y\in S}h_{\alpha}(\x^{(k+1)},\y,\cdots,\u^{(k)},\z^{(k)})+\frac{\gamma_2}{2}\|\y-\y^{(k)}\|^2,\\
&~~~~~~~~~~~~~~~~~~~~~\vdots~~~~~~~~~~~~~~~~~~~~~\vdots~~~~~~~~~~~~~~~~~~~~\vdots\\
&\u^{(k+1)}=\arg\min_{\u\in S}h_{\alpha}(\x^{(k+1)},\y^{(k+1)},\cdots,\u,\z^{(k)})+\frac{\gamma_{d-1}}{2}\|\u-\u^{(k)}\|^2,\\
&\z^{(k+1)}=\arg\min_{\z\in S}h_{\alpha}(\x^{(k+1)},\y^{(k+1)},\cdots,\u^{(k+1)},\z)+\frac{\gamma_{d}}{2}\|\z-\z^{(k)}\|^2.\\
\end{aligned}
\right.
$$
\STATE Let $\v^{(k+1)}$ be defined such that
$$
\v^{(k+1)}=\arg\min\{h_{\alpha}(\x^{(k+1)}),h_{\alpha}(\y^{(k+1)}),\cdots,h_{\alpha}(\u^{(k+1)}),h_{\alpha}(\z^{(k+1)})\}.
$$
\STATE ~~~~If $|h_{\alpha}(\v^{(k+1)})-h_{\alpha}(\v^{(k)})|<\epsilon$, stop and return the approximate solution $\v^{(k+1)}$.
\STATE ~~~~end if
\STATE end for
\end{algorithmic}
\end{algorithm}

\begin{remark}
{Theoretically, it is possible to set all \(\gamma_i = 0\), as this does not affect the monotonicity of the sequence. However, in practice, we found that choosing \(\gamma_i \in [1,5]\) significantly improves convergence speed and reduces the number of iterations. Additional numerical examples and details are provided in Section~\ref{Sec4}.
It is also worth noting that while the convergence proof requires \(\alpha \geq \|\A\|_F\) as a sufficient condition, this is not strictly necessary in practice. In fact, selecting \(\alpha < \|\A\|_F\) can still yield successful convergence under certain conditions. On the other hand, choosing an excessively large \(\alpha\) may slow down convergence, whereas selecting \(\alpha\) too small can lead to the failure of the method. Therefore, a careful selection of \(\alpha\) is crucial for balancing convergence speed and stability.}
\label{remark 2}
\end{remark}

\begin{remark}
(1) Algorithm \ref{alg32} is straightforward to implement under spherical constraints $\x \in S$. Since the linearly independent constraint qualification (LICQ) automatically holds, the (local) optimal solution for each subproblem of Algorithm \ref{alg32} is guaranteed to be a KKT point. Considering the $\x$-subproblem of Algorithm \ref{alg32},
\begin{eqnarray}
\label{subprob of PAM}
\x^{(k+1)} = \arg\min_{\x \in S} h_{\alpha}(\x, \y^{(k)}, \cdots, \u^{(k)}, \z^{(k)}) + \frac{\gamma_1}{2} \|\x - \x^{(k)}\|^2.
\end{eqnarray}
By applying the first-order optimality condition, this subproblem has closed form solution, yielding only two candidates
\begin{eqnarray}
\label{closed form sol}
\x^{(k+1)} = \pm \frac{\A(\theta)\y^{(k)}\cdots\u^{(k)}\z^{(k)} - \alpha \y^{(k)}\cdots\langle\u^{(k)},\z^{(k)}\rangle - \gamma_1 \x^{(k)}}{\|\A(\theta)\y^{(k)}\cdots\u^{(k)}\z^{(k)} - \alpha \y^{(k)}\cdots\langle\u^{(k)},\z^{(k)}\rangle - \gamma_1 \x^{(k)}\|}.
\end{eqnarray}
Note that if the denominators of \eqref{closed form sol} equal zero, we may set \(\x^{k+1} = \x^{k}\) and adjust the parameters \(\gamma_1\) and \(\alpha\). Comments regarding the appropriate size of \(\gamma_1\) and \(\alpha\) are provided in Remark~\ref{remark 2}. Similarly, the KKT points for the subproblems with blocks $\y, \z\cdots,$ can be derived analogously. Since all subproblems in Algorithm \ref{alg32} admit finite analytic solutions, the algorithm is computationally efficient for spherically constrained problems.

\noindent(2) Similar algorithms have been applied to various polynomial optimization problems in the literature \cite{CSLZ2012,Chen2022,Wang2015}. For instance, \cite{Chen2022} employs such an algorithm for a biquadratic optimization problem defined on two unit spheres, where only subsequential convergence is established. In contrast, the Maximum Block Improvement (MBI) algorithm \cite{CSLZ2012} updates one block per iteration, while the Block Improvement Method (BIM) \cite{Wang2015} updates two blocks in sequential order. Our proposed algorithm differs by updating all blocks sequentially, which often results in greater improvement in the objective function value for multilinear optimization problems, as demonstrated by computational results in Section 5.

\noindent(3) Furthermore, Step 3 of Algorithm \ref{alg32} can be viewed as a special case of the block coordinate descent (BCD) algorithm \cite{Tseng2001} with a cyclic update rule. For the proximal parameters $\gamma_i$ ($i=1, \ldots, \frac{d}{2}$), they can theoretically take any positive values. In our experiments in Section 5, we select $\gamma_i \in (0, 10]$.

\end{remark}

\subsection{Convergence of the PAM Algorithm}

In this subsection, we first establish the subsequential convergence of the sequence $\{\t^{(k)}\} = \{(\x^{(k)}, \y^{(k)}, \cdots, \u^{(k)}, \z^{(k)})\}$ generated by Algorithm \ref{alg32}. Then, leveraging the Kurdyka-{\L}ojasiewicz (KL) property, we establish the global sequence convergence and give an explicit convergence rate.

\begin{theorem}\label{thm2}{\rm (Subsequential Convergence)} Let $\{\t^{(k)}\}$ be an infinite sequence generated by Algorithm \ref{alg32}.
Let $\bar{\gamma}=\min\{\gamma_1,\ldots,\gamma_d\}$. Then the following statements hold.

{\rm(i)} For all $k$, the function sequence $\{h_{\alpha}(\t^{(k)})\}$ is nonincreasing and convergent. The sequence $\{\t^{(k)}\}$ satisfies
$\sum_{k=1}^{+\infty}\|\t^{(k+1)}-\t^{(k)}\|<+\infty.$

{\rm(ii)} Suppose $\bar{\t}$ is a cluster point of $\{\t^{(k)}\}$, then, $\bar{\t}$ is a KKT point of (\ref{e7}) and
$\lim_{k\rightarrow+\infty}h_{\alpha}(\t^{(k)})=h_{\alpha}(\bar{\t}).$
\end{theorem}
\proof
(i) By Algorithm \ref{alg32}, it follows that
$$
\begin{aligned}
&h_{\alpha}(\t^{(k+1)})+\frac{\gamma_d}{2}\|\z^{(k+1)}-\z^{(k)}\|^2\leq h_{\alpha}(\x^{(k+1)},\cdots,\u^{(k+1)},\z^{(k)}),\\
&h_{\alpha}(\x^{(k+1)},\cdots,\u^{(k+1)},\z^{(k)})+\frac{\gamma_{d-1}}{2}\|\u^{(k+1)}-\u^{(k)}\|^2\leq h_{\alpha}(\x^{(k+1)},\cdots,\u^{(k)},\z^{(k)}),\\
&~~~~~~~~~~~~~~~~~~~~~~~\vdots~~~~~~~~~~~\vdots~~~~~~~~~~~~\vdots\\
&h_{\alpha}(\x^{(k+1)},\y^{(k)}\cdots,\u^{(k)},\z^{(k)})+\frac{\gamma_1}{2}\|\x^{(k+1)}-\x^{(k)}\|^2\leq h_{\alpha}(\t^{(k)}),
\end{aligned}
$$
which implies
\begin{equation}\label{e9}
h_{\alpha}(\t^{(k+1)})+\frac{\bar{\gamma}}{2}\|\t^{(k+1)}-\t^{(k)}\|^2\leq h_{\alpha}(\t^{(k)}).
\end{equation}
Therefore, the sequence $\{h_{\alpha}(\t^{(k)})\}$ is nonincreasing. Since $h_{\alpha}(\t)$ is bounded on the compact set $S$, we know that $\{h_{\alpha}(\t^{(k)})\}$ is convergent, and the sequence $\{\t^{(k)}\}$ satisfies $\sum_{k=1}^{+\infty}\|\t^{(k+1)}-\t^{(k)}\|<+\infty$ from (\ref{e9}). The conclusion (i) follows.

\noindent(ii) For the sequence $\{\t^{(k)}\}$, it has cluster points by bounded conditions. Without loss of generality, suppose that $\{\t^{(k_j)}\}$ is a subsequence of  $\{\t^{(k)}\}$ with cluster point $\bar{\t}=(\bar{\x}, \bar{\y},\cdots,\bar{\u},\bar{\z})$. By the proof of (i), it's clear that $\lim_{k\rightarrow+\infty}h_{\alpha}(\t^{(k)})=h_{\alpha}(\bar{\t})$. On the other hand, from (i) again, it holds that
$$
\lim_{k\rightarrow+\infty}\|\t^{(k+1)}-\t^{(k)}\|=0.
$$
Then, we have that $\lim_{j\rightarrow+\infty}\t^{(k_j+1)}=\bar{\t}.$ Considering the KKT system of the subproblems in Algorithm \ref{alg32}, there are lagrange multipliers $\lambda_1,\ldots,\lambda_d\in\R$ satisfying
$$
\left\{
\begin{aligned}
&\nabla_{\x} h_{\alpha}(\x^{(k_j+1)},\y^{(k_j)},\cdots,\z^{(k_j)})+\gamma_1(\x^{(k_j+1)}-\x^{(k_j)})-\lambda_1\x^{(k_j+1)}=\0,\\
&\nabla_{\y} h_{\alpha}(\x^{(k_j+1)},\y^{(k_j+1)},\cdots,\z^{(k_j)})+\gamma_2(\y^{(k_j+1)}-\y^{(k_j)})-\lambda_2\y^{(k_j+1)}=\0,\\
&~~~~~~~~~~~~~~~~~~~~~~~\vdots~~~~~~~~~~~\vdots~~~~~~~~~~~~\vdots\\
&\nabla_{\z} h_{\alpha}(\x^{(k_j+1)},\y^{(k_j+1)},\cdots,\z^{(k_j+1)})+\gamma_d(\z^{(k_j+1)}-\z^{(k_j)})-\lambda_d\z^{(k_j+1)}=\0.\\
\end{aligned}
\right.
$$
Let $j\rightarrow+\infty$, and by the continuity of $h_{\alpha}(\t)$, it follows that
\begin{equation}\label{e10}
\left\{
\begin{aligned}
&\nabla_{\x} h_{\alpha}(\bar{\x},\bar{\y},\cdots,\bar{\z})-\lambda_1\bar{\x}=\0,\\
&\nabla_{\y} h_{\alpha}(\bar{\x},\bar{\y},\cdots,\bar{\z})-\lambda_2\bar{\y}=\0, \\
&~~~~~~~~~~~~~~\vdots~~~~~~~~~~~\vdots~~~~~~~~~\\
&\nabla_{\z} h_{\alpha}(\bar{\x},\bar{\y},\cdots,\bar{\z})-\lambda_d\bar{\z}=\0.\\
\end{aligned}
\right.
\end{equation}
By the fact that $\bar{\x},\bar{\y},\cdots,\bar{\z}\in S$, we obtain that
$$
\lambda_1=\lambda_2=\cdots=\lambda_d=h_{\alpha}(\bar{\x},\bar{\y},\cdots,\bar{\z})=h_{\alpha}(\bar{\t}).
$$
Combining this with (\ref{e10}), it holds that $\nabla h_{\alpha}(\bar{\t})-h_{\alpha}(\bar{\t})\bar{\t}=\0,$
and the desired result holds.
\qed

\begin{remark} As a side note, we present an interesting observation that we propose as a conjecture. While it is not directly related to the convergence theories discussed in this paper, we believe this result could offer valuable insights into the relationship between homogeneous polynomials and their corresponding multilinear functions.
Let \( h(\x) \) be a homogeneous polynomial of even degree \( d \), with the corresponding multilinear function denoted as \( h(\x, \y, \cdots, \u, \z) \). If \( h(\x) \) is a concave function, then wthen we conjecture that the following inequality holds
\[
\min\{h(\x), h(\y), \cdots, h(\u), h(\z)\} \leq h(\x, \y, \cdots, \u, \z),
\]
for \( \|\x\| = \|\y\| = \dots = \|\z\| = 1 \).
For the case \( d = 2 \), where \( h_{\alpha}(\x) = \x^\top A\x \) and \( A \) is a matrix, the conjecture clearly holds. Using the concavity property, for any \( \x, \y \in \R^n \),
$
(\x - \y)^\top A (\x - \y) \leq 0,
$
which implies that
$
2\min\{\x^\top A\x, \y^\top A\y\} \leq \x^\top A\x + \y^\top A\y \leq 2\x^\top A\y.
$
For \( d = 4 \), a similar result has been proven in \cite{Chen2022}. However, for the case of \( d > 4 \), this remains an open question to the best of our knowledge, and thus we state this as a conjecture.
\label{remark conjecture}
\end{remark}

To establish global sequence convergence, numerous classical results from the literature have been employed to guarantee it for descent algorithms, including proximal algorithms, forward-backward splitting algorithms, regularized Gauss-Seidel methods, and others \cite{Hedy2013,Hedy2010}. The Kurdyka-{\L}ojasiewicz (KL) property (defined in Appendix \ref{appendix KL prop}) plays a pivotal role in analyzing global sequential convergence.  Below, we review a framework for proving global sequential convergence.

\begin{lemma}\label{lema3}{\rm\cite{Hedy2013}}
Let $\varphi: \R^n \rightarrow \R \cup \{+ \infty\}$ be a proper lower semicontinuous function\footnote{More definitions can be found in Appendix \ref{appendix KL prop}}. Consider a sequence $\{\x^{(k)}\}_{k>0}$ with $\x^{(k)} \in \R^n$ satisfying the following three conditions:

\noindent{\rm(i)}(Sufficient decrease condition) There exists $a>0$ such that
$$
\varphi(\x^{(k+1)})+a\|\x^{(k+1)}-\x^{(k)}\|^2\leq\varphi(\x^{(k)})
$$
holds for any $k\in\mathbb{N}$.

\noindent{\rm(ii)}(Relative error condition) There exist $b>0$ and $\omega^{(k+1)}\in\partial\varphi(\x^{(k+1)})$ (the subdifferential
 of $\varphi$ evaluated at $\x^{(k+1)}$), such that
$$
\|\omega^{(k+1)}\|\leq b\|\x^{(k+1)}-\x^{(k)}\|
$$
holds for any $k\in\mathbb{N}$.

\noindent{\rm(iii)} (Continuity condition) There exist subsequence $\{\x^{(k_j)}\}$ and $\x^*$ such that
$$
\x^{(k_j)}\rightarrow\x^*~~{\rm and}~~\varphi(\x^{(k_j)}) \rightarrow \varphi(\x^*),~~{\rm as}~j\rightarrow+\infty
$$
If $\varphi$ satisfies the KL property at $\x^*$, then  $\0\in\partial\varphi(\x^*)$, and $$\sum_{k=1}^{+\infty}\|\x^{(k+1)}-\x^{(k)}\|^2<+\infty,~~\lim_{k\rightarrow+\infty}\x^{(k)}=\x^*.$$
\end{lemma}

By Lemma \ref{lema3}, to prove global convergence, we only need to verify the corresponding conditions (i)--(iii). To move on, denote the indicator function $\iota_{S}(\t)$ on $S\times S\cdots\times S$ such that
$$
\iota_{S}(\t)=\left\{
\begin{array}{ll}
0, &~~{\rm if}~\t\in S\times S\cdots\times S,\\
+\infty, &~~{\rm otherwise}.\\
\end{array}
\right.
$$
We now reformulate the multilinear programming (\ref{e7}) as an equivalent unconstraint programming:
\begin{equation}\label{e11}
\min_{\t\in S\times S\times\cdots\times S}h_{\alpha}(\t)=\min_{\t\in \R^n\times\R^n\times\cdots\times\R^n}h_{\alpha}(\t)+\iota_{S}(\t).
\end{equation}

\begin{remark}\label{rek3} Let $\varphi(\t)=h_{\alpha}(\t)+\iota_{S}(\t)$. Then, $\varphi(\t)$ holds the KL property automatically. From Theorem \ref{thm2}, we know that
the sufficient decrease condition and continuous condition are satisfied. Note that for all $\t\in S\times\cdots\times S$,
it holds that
$$
\partial\varphi(\t)=\partial\left(h_{\alpha}(\t)+\iota_{S}(\t)\right)=\nabla h_{\alpha}(\t)+\partial\iota_{S}(\t),
$$
and $\partial\iota_{S}(\t)=\{\lambda \t : \lambda\in\R\}$, where $\partial$ means the Fr\'{e}chet subdifferential \cite{ZL2022}. Therefore, it holds that $\omega^{(k+1)}=\nabla h_{\alpha}(\t^{(k+1)})-h_{\alpha}(\bar{\t})\t^{(k+1)}\in\partial\varphi(\t)$.
\end{remark}

By Lemma \ref{lema3} and Remark \ref{rek3} verify the conditions (i)--(iii) and we are ready to establish global sequence convergence.

\begin{proposition}\label{prop5} Let $\lambda^*=h_{\alpha}(\bar{\t})$, which is defined as in Theorem \ref{thm2}. Then, for all iterations $k\in\mathbb{N}$, there is an $M>0$ such that
$$
\|\omega^{(k+1)}\|=\|\nabla h_{\alpha}(\t^{(k+1)})-\lambda^*\t^{(k+1)}\|\leq M\|\t^{(k+1)}-\t^{(k)}\|.
$$
\end{proposition}
\proof By Algorithm \ref{alg32} and the proof of Theorem \ref{thm2}, it's clear that
\begin{equation}\label{e12}
\left\{
\begin{aligned}
&\|\nabla_{\x} h_{\alpha}(\x^{(k+1)},\y^{(k)},\cdots,\z^{(k)})-\lambda^*\x^{(k+1)}\|\leq \gamma_1\|\x^{(k+1)}-\x^{(k)}\|,\\
&\|\nabla_{\y}h_{\alpha}(\x^{(k+1)},\y^{(k+1)},\cdots,\z^{(k)})-\lambda^*\y^{(k+1)}\|\leq \gamma_2\|\y^{(k+1)}-\y^{(k)}\|,\\
&~~~~~~~~~~~~~~~~~~~~~~~\vdots~~~~~~~~~~~\vdots~~~~~~~~~~~~\vdots\\
&\|\nabla_{\z} h_{\alpha}(\x^{(k+1)},\y^{(k+1)},\cdots,\z^{(k+1)})-\lambda^*\z^{(k+1)}\|\leq\gamma_d\|\z^{(k+1)}-\z^{(k)}\|,\\
\end{aligned}
\right.
\end{equation}
where $\gamma_i, i\in[d]$ are nonnegative numbers in Algorithm \ref{alg32}. By the multilinear structure of $h_{\alpha}(\t)$ and the triangle inequality, we obtain that
$$
\begin{aligned}
&\|\nabla_{\x} h_{\alpha}(\t^{(k+1)})-\lambda^*\x^{(k+1)}\| \\
\leq&\|\nabla_{\x}  h_{\alpha}(\x^{(k+1)},\cdots,\z^{(k+1)})-\nabla_{\x}  h_{\alpha}(\x^{(k+1)},\cdots,\u^{(k+1)},\z^{(k)})\| \\
&+\|\nabla_{\x} h_{\alpha}(\x^{(k+1)},\cdots,\u^{(k+1)},\z^{(k)})-\nabla_{\x}  h_{\alpha}(\x^{(k+1)},\cdots,\u^{(k)},\z^{(k)})\|\\
&+\cdots \cdots\cdots \cdots\cdots \cdots\cdots \cdots\cdots \cdots\\
&+\|\nabla_{\x} h_{\alpha}(\x^{(k+1)},\y^{(k+1)},\cdots,\z^{(k)})-\nabla_{\x}  h_{\alpha}(\x^{(k+1)},\y^{(k)},\cdots,\u^{(k)},\z^{(k)})\| \\
&+\|\nabla_{\x}  h_{\alpha}(\x^{(k+1)},\y^{(k)},\cdots,\u^{(k)},\z^{(k)})-\lambda^*\x^{(k+1)}\|\\
\leq &m^1_1\|\z^{(k+1)}-\z^{(k)}\|+m^1_2\|\u^{(k+1)}-\u^{(k)}\|+\cdots\cdots\\
&+m^1_{d-1}\|\y^{(k+1)}-\y^{(k)}\|+\gamma_1\|\x^{(k+1)}-\x^{(k)}\|\leq m_1\|\t^{(k+1)}-\t^{(k)}\|,
\end{aligned}
$$
where $m_1,m^1_1,\cdots,m^1_{d-1}$ are positive constants. Take $m^1_1$ as an example. Because
\begin{align*}
\nabla_{\x}  h_{\alpha}(\x^{(k+1)},\cdots,\z^{(k+1)})-\nabla_{\x}  h_{\alpha}(\x^{(k+1)},\cdots,\u^{(k+1)},\z^{(k)})=
\\\A_{h}\y^{(k+1)}\cdots\u^{(k+1)}(\z^{(k+1)}-\z^{(k)})
\end{align*}
is a continuous function with all variables defined in compact set $S$, it is bounded. Similar to the proof above, there are positive constants $m_i, i\in[d]$ satisfying that, for each block of $\t^{(k+1)}$,
$$
\left\{
\begin{aligned}
&\|\nabla_{\x} h_{\alpha}(\t^{(k+1)})-\lambda^*\x^{(k+1)}\|\leq m_1\|\t^{(k+1)}-\t^{(k)}\|, \\
&\|\nabla_{\y} h_{\alpha}(\t^{(k+1)})-\lambda^*\y^{(k+1)}\|\leq m_2\|\t^{(k+1)}-\t^{(k)}\|, \\
&~~~~~~~~~~~~~~~~~~~~~~~\vdots~~~~~~~~~~~\vdots~~~~~~~~~~~~\vdots\\
&\|\nabla_{\z} h_{\alpha}(\t^{(k+1)})-\lambda^*\z^{(k+1)}\|\leq m_d\|\t^{(k+1)}-\t^{(k)}\|. \\
\end{aligned}
\right.
$$
Take $M=m_1+m_2+\cdots+m_d$.  Combining this with the definition of $\omega^{(k+1)}$, it follows that
$$
\begin{aligned}
\|\omega^{(k+1)}\|=&\|\nabla h_{\alpha}(\t^{(k+1)})-\lambda^*\t^{(k+1)}\|\\
\leq&\|\nabla_{\x} h_{\alpha}(\t^{(k+1)})-\lambda^*\x^{(k+1)}\|+\|\nabla_{\y} h_{\alpha}(\t^{(k+1)})-\lambda^*\y^{(k+1)}\| \\
&+\cdots\cdots+\|\nabla_{\z} h_{\alpha}(\t^{(k+1)})-\lambda^*\z^{(k+1)}\|\\
\leq& M\|\t^{(k+1)}-\t^{(k)}\|,
\end{aligned}
$$
and the desired result follows.
\qed

By Theorem \ref{thm2}, Lemma \ref{lema3} and Proposition \ref{prop5}, we have the following global sequential convergence.

\begin{theorem}\label{thm3}{\rm(Global sequence convergence)} Assume that the sequence $\{\t^{(k)}\}$ is an infinite sequence generated by Algorithm \ref{alg32}. Then, it converges globally to a KKT point $\bar{\t}$ of (\ref{e7}).
\end{theorem}

We now study the convergence rate of Algorithm \ref{alg32}.
First of all, we derive the KL exponent of the associated functions of $\varphi(\t)$.
To do this, we need the classical {\L}ojasiewicz gradient inequality for polynomials (See Theorem 4.2 of \cite{DK2005} for the detail).
\begin{lemma}\label{lema4}{\rm({\L}ojasiewicz gradient inequality)}
Let $f$ be a polynomial of $\R^n$ with degree $d\in\mathbb{N}$. Suppose that $f(\bar{\x})=0$. Then there exist constants $\epsilon, c>0$ such that, for all $\x\in\R^n$ with $\|\x-\bar{\x}\|\leq\epsilon$, we have
$$
\|\nabla f(\x)\|\geq c|f(\x)|^{1-\tau},~{\rm where}~\tau=R(n,d)^{-1},~R(n,d)=\left\{
\begin{array}{ll}
1,&{\rm if}~~ d=1,\\
d(3d-3)^{n-1},& {\rm if}~~ d\geq2.
\end{array}
\right.
$$
\end{lemma}

By Lemma \ref{lema4}, the following theorem establishes the explicit KL exponent of the merit function.
\begin{theorem}\label{thm4} Let $\varphi(\t)=h_{\alpha}(\t)+\iota_{\Lambda}(\t)$ be defined as in (\ref{e11}), where
$\t=(\x^\top,\y^\top,\cdots,\u^\top$, $\z^\top)^\top\in \Lambda=S\times S\times\cdots\times S\subseteq\R^{dn}$. Assume that $d\geq2$. Then
$\varphi(\t)$ satisfies the KL property with exponent $1-\tau$ at $\bar{\t}=(\bar{\x}^\top,\bar{\y}^\top,\cdots,\bar{\u}^\top, \bar{\z}^\top)^\top$,
where $\tau=d^{-1}(3d-3)^{1-dn}$.
\end{theorem}

\proof To prove the statement, let $\bar{\t}\in \Lambda$, and let $\delta_1, \eta>0$ such that for any $\t$ satisfying $\|\t-\bar{\t}\|\leq \delta_1$, it follows that
$$
\varphi(\bar{\t})\leq \varphi(\t)\leq\varphi(\bar{\t})+\eta.
$$
On one hand, we can write
$$
\varphi(\t)=h_{\alpha}(\x,\y,\cdots,\u,\z)+\iota_{S}(\x)+\iota_{S}(\y)+\cdots+\iota_{S}(\z).
$$
By a direct computation, we know that
\begin{equation}\label{e13}
\left\{
\begin{aligned}
&\partial_{\x}\varphi(\t)=\{\nabla_{\x}h_{\alpha}(\t)+\lambda_1\x : \lambda_1\in\R\},\\
&\partial_{\y}\varphi(\t)=\{\nabla_{\y}h_{\alpha}(\t)+\lambda_2\y : \lambda_2\in\R\},\\
&~~~~~~~~~~~~~~~\vdots~~~~~~~~~~~~~~~\vdots\\
&\partial_{\u}\varphi(\t)=\{\nabla_{\u}h_{\alpha}(\t)+\lambda_{d-1}\u : \lambda_{d-1}\in\R\},\\
&\partial_{\z}\varphi(\t)=\{\nabla_{\z}h_{\alpha}(\t)+\lambda_d\z : \lambda_d\in\R\}.\\
\end{aligned}
\right.
\end{equation}
By (\ref{e13}), it implies that
\begin{equation}\label{e14}
\begin{aligned}
{\rm dist}(0, \partial\varphi(\t))^2
&=\inf_{\lambda_1,\cdots,\lambda_d\in\R}\|\partial_{\x}\varphi(\t)\|^2+\|\partial_{\y}\varphi(\t)\|^2+\cdots+\|\partial_{\z}\varphi(\t)\|^2,\\
&=\inf_{\lambda_1\in\R}(\lambda_1^2+2\lambda_1h_{\alpha}(\t)+\|\nabla_{\x}h_{\alpha}(\t)\|^2)\\
&~~+\inf_{\lambda_2\in\R}(\lambda_2^2+2\lambda_2h_{\alpha}(\t)+\|\nabla_{\y}h_{\alpha}(\t)\|^2)+\cdots\\
&~~+\inf_{\lambda_d\in\R}(\lambda_d^2+2\lambda_dh_{\alpha}(\t)+\|\nabla_{\z}h_{\alpha}(\t)\|^2)\\
&=-dh_{\alpha}(\t)^2+\|\nabla h_{\alpha}(\t)\|^2.
\end{aligned}
\end{equation}

On the other hand, we consider the following polynomial
$$
f(\x,\y,\cdots,\u,\z,\lambda)=h_{\alpha}(\t)+\frac{\lambda}{2}(\|\x\|^2-1)+\frac{\lambda}{2}(\|\y\|^2-1)
+\cdots+\frac{\lambda}{2}(\|\z\|^2-1),$$
where $\lambda=-h_{\alpha}(\t)=-h_{\alpha}(\x,\y,\cdots,\z)$. Denote
$$\hat{f}(\x,\y,\cdots,\z,\lambda)=f(\x,\y,\cdots,\z,\lambda)-f(\bar{\x},\bar{\y},\cdots,\bar{\z},\bar{\lambda}),$$
where $\bar{\lambda}=-h_{\alpha}(\bar{\x},\bar{\y},\cdots,\bar{\z})$.
Obviously that $\hat{f}(\t)$ is a polynomial defined in $\R^{nd}$ with degree $d$.
By Lemma \ref{lema4}, there exist $\delta'>0, c>0$
such that, for all $\|\t-\bar{\t}\|\leq\delta_2$, it follows
$$
\begin{aligned}
\|\nabla f(\x,\y,\cdots,\z,\lambda)\|&=\|\nabla \hat{f}(\x,\y,\cdots,\z,\lambda)\|\\
&\geq c|f(\x,\y,\cdots,\z,\lambda)-f(\bar{\x},\bar{\y},\cdots,\bar{\z},\bar{\lambda})|^{1-\tau},
\end{aligned}
$$
where $\tau=d^{-1}(3d-3)^{1-dn}$. Note that for any $\t\in\Lambda, \lambda\in\R$, we have
$$
\left\{
\begin{aligned}
&\nabla_{\x} f(\x,\y,\cdots,\z,\lambda)=\nabla_{\x}h_{\alpha}(\t)+\lambda\x,\\
&\nabla_{\y} f(\x,\y,\cdots,\z,\lambda)=\nabla_{\y}h_{\alpha}(\t)+\lambda\y,\\
&~~~~~~~~~~~~~~\vdots~~~~~~~~~~~~~~\vdots\\
&\nabla_{\z} f(\x,\y,\cdots,\z,\lambda)=\nabla_{\z}h_{\alpha}(\t)+\lambda\z,\\
&\nabla_{\lambda} f(\x,\y,\cdots,\z,\lambda)=0,\\
\end{aligned}
\right.
$$
which implies that
$$
\|\nabla f(\x,\y,\cdots,\z,\lambda)\|^2=-dh_{\alpha}(\t)^2+\|\nabla h_{\alpha}(\t)\|^2={\rm dist}(0, \partial\varphi(\t))^2,
$$
and $f(\x,\y,\cdots,\z,\lambda)=\varphi(\t)$, $f(\bar{\x},\bar{\y},\cdots,\bar{\z},\bar{\lambda})=\varphi(\bar{\t})$.
Take $\delta=\min\{\delta_1, \delta_2\}$. Combining this with (\ref{e14}), it holds hat, for all $\t\in\Lambda$ with $\|\t-\bar{\t}\|\leq\delta$,
and $\varphi(\bar{\t})\leq \varphi(\t)\leq\varphi(\bar{\t})+\eta$,
$$
{\rm dist}(0, \partial\varphi(\t))\geq c|f(\x,\y,\cdots,\z,\lambda)-f(\bar{\x},\bar{\y},\cdots,\bar{\z},\bar{\lambda})|^{1-\tau},
$$
and the desired results hold.
\qed

To conclude this section, we analyze the convergence rate of Algorithm \ref{alg32} for the non trivial cases of $d\ge 2$ and $n\ge 2$. Under reasonable assumptions, several classical results on convergence rates from the literature established the convergence rate analyses based on the KL property \cite{Hedy2009,WCP18,YY2013}. Specifically, if the desingularization function of $\varphi(\t)$ is $\phi(s)=cs^{1-\alpha}$, then as shown in \cite{Hedy2009}, the following estimates hold
\begin{itemize}
    \item[(i)] If $\alpha=0$, the sequence $\{\t^{(k)}\}$ converges in a finite number of steps.
    \item[(ii)] If $\alpha\in(0, 1/2]$, there exist constants $a>0$ and $\theta\in(0, 1)$ such that $\|\t^{(k)}-\bar{\t}\|\leq a\theta^k$.
    \item[(iii)] If $\alpha\in(1/2, 1)$, there exists a constant $a>0$ such that
    $$
    \|\t^{(k)}-\bar{\t}\|\leq ak^{-\frac{1-\alpha}{2\alpha-1}}.
    $$
\end{itemize}
Combining this with Theorem \ref{thm4} (where $\alpha=1-\tau$), we obtain the following result.

\begin{theorem}\label{thm5}
Assume $\{\t^{(k)}\}$ is an infinite sequence generated by Algorithm \ref{alg32}, and $\lim_{k\rightarrow\infty}\t^{(k)}=\bar{\t}$. For $d\geq 2$ and $n\geq 2$, there exists a constant $a>0$ such that
$$
\|\t^{(k)}-\bar{\t}\|\leq ak^{-\tau/(1-2\tau)},
$$
where $\tau=d^{-1}(3d-3)^{1-dn}$ is defined in Theorem \ref{thm4}.
\end{theorem}

\proof
By Theorem \ref{thm4}, the merit function $\varphi(\t)$ satisfies the KL property with exponent $1-\tau$ at $\bar{\t}$. Hence, it suffices to verify that $1-\tau\in(1/2, 1)$. From the definition of $\tau$, it is evident that
$$
\tau=d^{-1}(3d-3)^{1-dn}<1/2,~\forall~d\geq 2,~n\geq 2,
$$
which ensures the desired result.
\qed

\section{Applications and Preliminary Numerical Results}\label{Sec4}

In this section, we apply the proposed Algorithm \ref{alg32} (under the framework of Algorithm \ref{alg31}) to two applications. The first application involves computing the extremal generalized tensor eigenpair, where we compare our algorithm with state-of-the-art methods from Kolda et al. \cite{Kolda2014}.
The second application utilizes Algorithm \ref{alg32} to minimize the high-order trust-region method \cite{Cartis2022}.

\subsection{Application to Extremal Generalized Tensor Eigenpair}


Let $\A$ and $\B$ be real-valued, $m$-th order, $n$-dimensional symmetric tensors.
Then, $(\lambda, \x) \in \R \times (\R^n \setminus \{\0\})$ is called a generalized eigenpair (also known as a $\B$-eigenpair) if the following holds:
\begin{equation}\label{e15}
\A\x^{m-1} = \lambda\B\x^{m-1}.
\end{equation}
By the homogeneity of \eqref{e15}, we always assume that the corresponding eigenvectors lie on the unit sphere. Additionally, we assume that \(m\) is even and \(\B\) is positive definite, i.e., \(\B\x^d > 0\) for all \(\x \in \R^n \setminus \{\0\}\). These assumptions are necessary, as the existence of generalized eigenpairs cannot be guaranteed otherwise. However, for certain special cases, such as $Z$-eigenpairs, \(m\) can be odd. A comparison among B-, Z-, H-, and D-eigenvalues is given in Appendix \ref{Appendix: Def of Tensor eig} and we also refer readers to Chang et al. \cite{CPT2009} and Kolda et al. \cite{Kolda2014}

To the best of our knowledge, the shifted symmetric high-order power method is one of the most effective numerical methods for finding $Z$-eigenpairs of symmetric tensors \cite{Kolda2011}. Later, Kolda et al. \cite{Kolda2014} extended this method to the generalized eigenproblem adaptive power (GEAP) method for computing the generalized eigenpairs of \eqref{e15}. In this section, we compare the performance of the PAM algorithm with the GEAP algorithm. Notably, we primarily focus on the minimal eigenpair case, i.e., where the parameter \(\beta = -1\) in GEAP \cite{Kolda2014}.

We tested  Algorithm \ref{alg32} in MATLAB R2023b and the Tensor Toolbox Version 3.6 \cite{Bader2012}. 
All tests were conducted on an Intel(R) Core(TM) i7-4770 CPU @ 3.40GHz 3.40 GH processor with 16 GB of RAM. The code is written in double precision. The tolerances for Algorithms \ref{alg31} and \ref{alg32} are set to $(\texttt{TOL}, \epsilon) = (10^{-3}, 10^{-6})$, respectively. 

The following example is originally from \cite{Kofidis2002} and was used to evaluate the shifted symmetric high-order power algorithm in \cite{Kolda2011}. Here, we focus on computing the minimal $Z$-eigenpair in Example \ref{example2} (Example 5.1 in \cite{Kolda2014}).

\begin{example}\label{example2} \textbf{(Computing  $Z$-eigenpair)}
Let $\A = (a_{i_1i_2i_3i_4})$, $i_1, i_2, i_3, i_4 \in [3]$, be a symmetric fourth-order three-dimensional tensor with all nonzero entries as detailed in Appendix \ref{appendix coeff}. Our objective is to compute the smallest $Z$-eigenpair of $\A$, let $\B = \mathcal{E}$, where $\mathcal{E}$ is the identity tensor such that $\mathcal{E}\x^{m-1} = \|\x\|^{m-2}\x$ for all $\x \in \R^n$ \cite{CPT2009}.
\end{example}

Tables \ref{table1}--\ref{table for example 2} show the results of 100 runs using random initial guesses, with each entry selected uniformly randomly from the interval $[-1, 1]$; the same set of random starts was used for each set of experiments. For each eigenpair, the table lists the number of occurrences in the 100 experiments, the eigenvalue found in each run along with its standard deviation and eigenvectors, the mean number of iterations for PAM until convergence with its standard deviation, and the average error and standard deviation. We tested Example \ref{example2} with various values of $\alpha$ and $\{\gamma_i\}_{1 \le i \le d}$.

For the tested parameters, all problems converged to a $Z$-eigenpair of the tensor system with an accuracy of $10^{-6}$, satisfying $\A[\x]^4 = \lambda \x$.
Notably, in all experiments, the outer iteration (i.e., the Dinkelbach--Type Algorithm \ref{alg31}) converged in a single iteration, demonstrating that the PAM algorithm solves the subproblem to a local minimum with sufficient accuracy.
Overall, Algorithm \ref{alg32} achieved the global minimum (i.e., the smallest eigenvalue $\lambda = -1.095$) in a greater number of instances compared to GEAP in \cite{Kolda2014}, with approximately 50-60\% convergence versus 40\% reported in Kolda  \cite{Kolda2014}. With larger $\gamma_i$, the algorithm exhibits an increasing tendency to locate the global minimum.
The iteration count and CPU time are comparable to GEAP. However, since the two methods were computed using tolerance, different software versions, and computational setups, further investigation is required to ensure a fair comparison under identical conditions.

\begin{table}[!h]
\centering
  \caption{\small  Example \ref{example2}: Parameters are set at $\gamma_1 =\dotsc=\gamma_d = 1,5$ from top to bottom, $\alpha = \|A\|_F$.  }
\centering
\label{table1}
\begin{tabular}{|c|c|c|c|c|}
\hline
\textbf{occ.(\%)} & $\lambda$ & $\mathbf{x}$ & $its$ & \text{CPU time}   \\[1ex]
\hline
42 &  $-1.0954 $  & $\pm[0.5916 , -0.7461,-0.3045]^\top $ & $16.8 \pm 3.3$ & $0.055 \pm 0.011$ \\*[1ex]
\hline
33 &  $-0.5629 $  &  $\pm[0.1771 , -0.1793, 0.9677]^\top $& $18.3 \pm 3.4$  & $0.061 \pm 0.011$\\*[1ex]
\hline
25 &   $-0.0451 $   &  $\pm[0.7797 , 0.6136 , 0.1246]^\top $ & $32.0 \pm 4.3$ & $0.105 \pm 0.014$ \\*[1ex]
\hline
\end{tabular}

\bigbreak
\begin{tabular}{|c|c|c|c|c|}
\hline
\textbf{occ. (\%)} & $\lambda$ & $\mathbf{x}$ & $its$ & \text{CPU time}   \\[1ex]
\hline
43 &  $-1.095$  & $\pm[0.5916 , -0.7461, -0.3045]^\top$ & $34.6 \pm 9.4$ & $0.121 \pm 0.034$ \\*[1ex]
\hline
36 &  $-0.5631$  & $\pm[0.1771 , -0.1793, 0.9677]^\top$ & $37.2 \pm 6.5$ & $0.130 \pm 0.029$ \\*[1ex]
\hline
21 &  $-0.0450 $  & $\pm[0.7797 , 0.6136 , 0.1246]^\top$ & $68.6 \pm 17.2$ & $0.238 \pm 0.071$ \\*[1ex]
\hline
\end{tabular}


\end{table}

\begin{table}[!h]
\centering
  \caption{\small  Example \ref{example2}: Parameters are set at $\gamma_1 =\dotsc=\gamma_d = 1$, $\alpha = 10, 0.1$ from top to bottom.  }
\centering
\begin{tabular}{|c|c|c|c|c|}
\hline
\textbf{occ. (\%)} & $\lambda$ &$x$ & $its$ & \text{CPU time}   \\[1ex]
\hline
43 & $-1.0953 $  & $\pm[0.5916 , -0.7461,- 0.3045 ]^\top $ & $25 \pm 5 $ & $ 0.05\pm 0.01$ \\*[1ex]
\hline
38 & $-0.5629  8$   &  $ \pm [0.1771 , -0.1793, 0.9677]^\top $& $28 \pm 6$  & $0.07\pm 0.01 $\\*[1ex]
\hline
19 &  $-0.0451  $  &  $ \pm[0.7797   , 0.6136  ,  0.1246]^\top $ & $53\pm 6$ & $ 0.12\pm 0.01 $\\*[1ex]
\hline
\end{tabular}
\bigbreak
\begin{tabular}{|c|c|c|c|c|}
\hline
\textbf{occ. (\%)} & $\lambda$ &$x$ & $its$ & \text{CPU time}   \\[1ex]
\hline
70 & $-1.0952 $ & $\pm[0.5916 , -0.7461,- 0.3045 ]^\top $ & $22 \pm 6 $ & $ 0.06\pm 0.02$ \\*[1ex]
\hline
11 & $ -0.5629   $ &  $ \pm [0.1771 , -0.1793, 0.9677]^\top $& $99 \pm 51$  & $0.24\pm 0.11 $\\*[1ex]
\hline
19 &   $0.8663 $&  $ \pm[0.7797   , 0.6136  ,  0.1246]^\top $ & $40 \pm 16$ & $ 0.10\pm 0.04 $\\*[1ex]
\hline
\end{tabular}
\label{table for example 2}
\end{table}

The following example is adapted from \cite{Kolda2014} for computing the H-eigenpairs of a symmetric tensor (Example 5.2 in \cite{Kolda2014}).

\begin{example}\label{example3}
Let $\A=(a_{i_1i_2i_3i_4i_5i_6}), i_j\in[4], j\in[6]$ be a sixth-order four-dimensional symmetric tensor with nonzero entries defined in  Appendix \ref{appendix coeff}.
Denote tensor $\B$ with entries such that
$$
b_{i_1i_2i_3i_4i_5i_6}=\left\{
\begin{array}{ll}
1,&{\rm if}~i_1=i_2=i_3=i_4=i_5=i_6,\\
0,&{\rm otherwise}.
\end{array}
\right.
$$
Therefore, for any $\x\in\R^4$, $\B\x^5=\x^{[5]}=(x_1^5,x_2^5,x_3^5,x_4^5)^\top$.
\end{example}

Table \ref{table for example 3} presents the eigenpairs computed by Algorithm \ref{alg32} within the framework of Algorithm \ref{alg31} over 100 random trials. The initialization is chosen as random numbers ranging from \(0\) to \(1\).  
For all trials, the algorithm converges within the error bounds \(\epsilon = 10^{-6}\) and \(\texttt{TOL} = 10^{-3}\). We identify three eigenvalues (\(-3.7082\), \(-2.0798\), and \(-1.9568\)), which represent local minima. The standard deviation for each random initialization is below \(10^{-12}\) in all trials. All trials are verified to converge to an eigenvalue satisfying \(\A[x]^5 = \B[x]^5\), with convergence achieved in 6 outer iterations and approximately 200–400 inner iterations. Both outer and inner iterations exhibit monotonically decreasing error values.

\begin{table}[!h]
\centering
  \caption{\small Example \ref{example3}: Parameters are set at $\gamma_1 =\dotsc=\gamma_d = 3$, $\alpha = 3,$ averaged over 100 random trials.  }
\centering
\begin{tabular}{|c|c|c|c|c|}
\hline
\textbf{Occ. (\%)} & $\lambda$ & \textbf{Inner Its.} & \textbf{Outer Its.} & \textbf{CPU Time (s)} \\[1ex]
\hline
23 & $-3.7082 \pm 1 \times 10^{-15}$ & $243.30 \pm 5.09$ & $6 \pm 0.00$ & $3.279 \pm 0.214$ \\*[1ex]
\hline
30 & $-2.0798 \pm 1 \times 10^{-15}$ & $419.83 \pm 5.76$ & $6 \pm 0.00$ & $5.774 \pm 0.340$ \\*[1ex]
\hline
47 & $-1.9568 \pm 2.101 \times 10^{-12}$ & $310.36 \pm 22.28$ & $6 \pm 0.00$ & $4.189 \pm 0.306$ \\*[1ex]
\hline
\end{tabular}

\label{table for example 3}
\end{table}

Next, an example is adapted from \cite{Kolda2014} for computing the D-eigenpair problem\footnote{More details about the D-eigenpair problem can be found in Appendix \ref{Appendix: Def of Tensor eig}}.

\begin{example}\label{example4} The problem was proposed by Qi, Wang, and Wu \cite{Qi2008} for diffusion kurtosis imaging (DKI)\footnote{Note that only four digits of precision for the tensors are provided in \cite{Kolda2014,Qi2008}. Similar to \cite{Kolda2014}, we were unable to validate the solutions reported in the original paper \cite{Qi2008}. It is unclear whether this discrepancy is due to a lack of precision or a typographical error in the paper. Similar to  \cite{Kolda2014}, the problem is also rescaled. The results in Table 5.3 of \cite{Kolda2014} correspond to $\mathcal{D}$ and $\mathcal{A}$ of order $[2,3]$, and the result in \ref{table for example 4} correspond to $\mathcal{A}$ and $\mathcal{B}$ in Appendix \ref{appendix coeff}. }
. We consider this example here since it can be expressed as a generalized tensor eigenproblem  (Example 5.3 in \cite{Kolda2014}). $\A, \B$ are symmetric tensors with order 4 and dimension 3 with entries given in Appendix \ref{appendix coeff}.
\end{example}

Table \ref{table for example 4} presents the eigenpairs computed by Algorithm \ref{alg32} under the framework of Algorithm \ref{alg31}. For all trials, the algorithm converges within the error bound \(\|\A \x^{m-1} - \lambda B \x^{m-1}\|_2 \le \texttt{TOL}\). We identify three eigenvalues from our random trials. Unlike the Z-eigenvalue case in Example \ref{example4}, where convergence occurs in a single outer iteration of the Dinkelbach--Type Algorithm \ref{alg31}, the computation of D-eigenvalues typically requires 3-5 outer iterations. Notably, we observe that the function values generated by Step 4 of Algorithm \ref{alg32} decrease monotonically, as expected.

\begin{table}[!h]
\centering
  \caption{\small Example \ref{example4}: Parameters are set at $\gamma_1 =\dotsc=\gamma_d = 1$, $\alpha = 10,$ averaged over 80 random trials.  }
\centering
\begin{tabular}{|c|c|c|c|c|}
\hline
\textbf{Occ. (\%)} & $\lambda$ & \textbf{Inner Its.} & \textbf{Outer Its.} & \textbf{CPU Time (s)} \\[1ex]
\hline
42.5 & $-0.2268 \pm 2.19 \times 10^{-9}$ & $119.41 \pm 26.26$ & $4.38 \pm 0.65$ & $0.795 \pm 0.185$ \\*[1ex]
\hline
28.75 & $-0.1241 \pm 1.09 \times 10^{-9}$ & $110.78 \pm 18.76$ & $4.22 \pm 0.52$ & $0.738 \pm 0.131$ \\*[1ex]
\hline
28.75 & $-0.0426 \pm 8.27 \times 10^{-13}$ & $106.43 \pm 14.85$ & $5.00 \pm 0.00$ & $0.871 \pm 0.058$ \\*[1ex]
\hline
\end{tabular}
\label{table for example 4}
\end{table}

\subsection{Minimizing the High-order Trust-region Subproblem}

In contrast to first-order and second-order methods, high-order methods (with order $p \geq 3$) have garnered significant research interest due to their faster global and local convergence rates \cite{Cartis2020b}. However, efficiently solving the associated subproblem, which involves a $p$th-order Taylor derivative, remains an open question for $p\ge 3$. Algorithm \ref{alg32} represents an initial attempt to minimize such subproblems effectively under a tensor eigenvalue framework.

We consider the unconstrained nonconvex optimization problem
\[
\min_{\x \in \R^n} \tilde{f}(\x),
\]
where $\tilde{f} : \R^n \to \R$ is nonconvex, $p$-times continuously differentiable ($p \geq 1$) and bounded below.
One of the key techniques in optimization involves approximating $\tilde{f}(\x+\s)$ at the current iterate $\x=\x_k$ using a $p$th-order Taylor expansion \( T_p(\s) \). This expansion is then minimized iteratively under the trust-region constraint \( \|\s\| \leq \Delta \), where \(\Delta > 0\)
\begin{equation}
\argmin_{\|\s\| \leq \Delta} T_p(\s) := \tilde{f}(\x_k) + \sum_{j=1}^p \frac{1}{j!} \nabla_{\x}^j \tilde{f}(\x_k)[\s]^j.
\label{subprob}
\end{equation}
Here, $\s \in \R^n$ and $\nabla_{\x}^j \tilde{f}(\x_k)$ is a symmetric $j$th-order tensor.

For $p=1$, \eqref{subprob} reduces to the steepest descent model with step size control, while for $p=2$, it gives the trust-region model, both of which have well-established minimization algorithms. However, for $p \geq 3$, \eqref{subprob} leads to (nonconvex) high-order trust-region models \cite{Cartis2022},  whose minimization remains an open research question. While recent literature works focus on solving the high-order subproblem under an adaptive regularization setup \cite{Cartis2023,Nesterov2021,Zhu2023,Zhu2024,Zhu2022,Cartis2024Efficient,Zhu2024Global},
in this subsection, we give preliminary numerical illustrations of using the Proximal Alternating Minimization (PAM) algorithm to minimize the $p$th-order Taylor model on the trust-region boundary \(\|\s\| = \Delta\).

The first step involves expressing the Taylor model \eqref{subprob} in the form of a symmetric tensor expression. To achieve this, we represent $T_p(\s) $ as a polynomial
\[
T_p(\s) = \sum_{\alpha \in \mathbb{Z}_+^n} f_\alpha \s^\alpha,
\]
where \(\s = [s_1, \dotsc, s_n]^\top\),
\[
\s^\alpha = s_1^{\alpha_1}s_2^{\alpha_2}\cdots s_n^{\alpha_n}, \quad |\alpha| = \alpha_1 + \alpha_2 + \cdots + \alpha_n \leq p,
\]
and \(\alpha_1, \dotsc, \alpha_n \in \{0, \dotsc, p\}\). Define a symmetric tensor \(\T \in \R^{(n+1)^{p}}\) with entries
\begin{equation}
\T\big\{\pi(i_1, \dotsc, i_{p})\big\} = \frac{(p-|\alpha|)!\alpha_1!\cdots\alpha_n!}{p!} f_\alpha, \quad i_1, \dotsc, i_{p} \in \{0, 1, \dotsc, n\},
\label{M}
\end{equation}
where \(\T\big\{i_1, \dotsc, i_{p}\big\}\) denotes the \([i_1, \dotsc, i_{p}]\)-th entry of the tensor \(\T\), \(\pi(i_1, \dotsc, i_{p})\) is the permutation of indices, and \(f_\alpha\) is the coefficient of \(\s^\alpha\). Using this definition, the Taylor model can be reformulated as the symmetric tensor expression,
\begin{eqnarray}
\label{tensor reform}
T_p(\s) = \T[\s_c]^{p} = \sum_{i_1, \dotsc, i_{p}} \T\big\{\pi(i_1, \dotsc, i_{p})\big\} s_{i_1} \dotsc s_{i_{p}},
\end{eqnarray}
where \(\s_c = [1, s_1, \dotsc, s_n]^\top \in \R^{n+1}\).

The minimization of $T_p$ on the boundary can be reformulated as a tensor eigenvalue problem.
If the minimizer of \eqref{subprob} lies on the boundary \(\|\s\| = \Delta\), \eqref{subprob} transforms into the Z-eigenvalue problem 
\begin{equation}
\argmin_{\s \in \R^{n+1}} \T[\s_c]^{p} \quad \text{s.t.} \quad \|\s\| = \Delta
\label{gep z}
\end{equation}
where $\s_c = [1, \s]^\top = [1, s_1, \dotsc, s_n]^\top \in \R^{n+1}$. Due to the constraint that the first entry of \(\s_c\) is fixed at 1. The closed-form expression for the KKT point in the PAM subproblem \eqref{closed form sol} becomes
\begin{eqnarray}
\label{closed form sol TR}
\x^{(k+1)} = \pm \Delta \frac{(\0,I) \big[\T\tilde{\y}^{(k)}\cdots\tilde{\u}^{(k)}\tilde{\z}^{(k)} - \alpha \tilde{\y}^{(k)}\cdots\langle\tilde{\u}^{(k)},\tilde{\z}^{(k)}\rangle - \gamma_1 \tilde{\x}^{(k)}\big]}
{\bigg\|(\0,I) \big[\T\tilde{\y}^{(k)}\cdots\tilde{\u}^{(k)}\tilde{\z}^{(k)} - \alpha \tilde{\y}^{(k)}\cdots\langle\tilde{\u}^{(k)},\tilde{\z}^{(k)}\rangle - \gamma_1 \tilde{\x}^{(k)} \big] \bigg\|}.
\end{eqnarray}
Similar to \eqref{closed form sol}, if the denominators of \eqref{closed form sol TR} equal zero, we may set \(\x^{k+1} = \x^{k}\) and adjust the parameters \(\gamma_1\) and \(\alpha\). 
Here, \(\tilde{\y} = (1, \y^\top)^\top\), \(\tilde{\u} = (1, \u^\top)^\top\), \(\tilde{\z} = (1, \z^\top)^\top\), \(\tilde{\x} = (1, \x^\top)^\top \in \R^{n+1}\), and \(I \in \R^{n \times n}\) is the identity matrix and $(\0,I)\in\mathbb{R}^{n\times(n+1)}$. 
Let $\s^*$ be the solution of the Z-eigenvalue problem 
in \eqref{gep z}. {Let \(\lambda^* = (\s^*)^\top \nabla T_3(\s^*) / \Delta^2\), and define the Lagrangian multiplier as  $$
L(\s, \lambda) := T_p(\s) + \frac{1}{2} \lambda \|\s\|^2.
$$
The gradient of the Lagrangian is zero at \((\lambda^*, \s^*)\), such that  
$$
\nabla L(\s^*, \lambda^*) := \nabla T_p(\s^*) + \lambda^* \s^* = \textbf{0}.
$$
We provide the algorithm for minimizing the high-order trust-region subproblem in Algorithm~\ref{alg33}.}

\begin{algorithm}[!htbp]
\caption{Algorithm for Minimizing the High-order Trust Region Subproblem}\label{alg33}
\begin{algorithmic}[1]
\STATE \textbf{Input:} $\Delta > 0$, the Taylor polynomial \(T_p(\s)\), its derivative \(\nabla T_p(\s)\). Set a tolerance level \texttt{TOL} and a maximum iteration count \(k_{\max}\). Initialize   \(\s_0 = \textbf{0} \in \R^n\), \(\lambda_0 = 0\), \(k = 0\) and $
\nabla L(\s_0, \lambda_0) := \nabla T_p(\s_0) + \lambda_0 \s_0$.\medskip

\STATE Generate the corresponding homogeneous tensor formulation \(\T \in \R^{(n+1)^p}\) using \eqref{tensor reform}.
 
\STATE While \( \|\nabla_{\s} L(\s_k, \lambda_k)\| = \|\nabla T_p(\s_k) + \lambda_k \s_k\| \geq \texttt{TOL} \) and \(k \leq k_{\max}\), perform Steps 4-5. Otherwise, terminate and return \((\s_k, \lambda_k)\). 

\STATE Solve the subproblem using Algorithm~\ref{alg32}, where the solution is given in the form \eqref{closed form sol TR} for the Z-eigenvalue problem in \eqref{gep z}. Return \(\s_k\) as the output of this step.\medskip

\STATE Update  
\[
\lambda_{k} = \frac{{\s_k}^\top \nabla T_p(\s_k)}{\Delta^2}.
\]
Set \(k := k + 1\) and return to Step 2.
\end{algorithmic}
\end{algorithm}

\subsubsection{Preliminary Numerical Results}

\textbf{Numerical Set-up}: We construct third-order Taylor polynomials,
\[
T_{3}(\s) = f_0 + g^\top\s + \frac{1}{2} H [\s]^2 + \frac{1}{6} T [\s]^3,
\]
to test the PAM algorithm. Specifically, the coefficients \(g\), \(H\), and \(T\) are generated as follows:
\[
g = \texttt{a*randn(n, 1)},  H = \texttt{b*symm(randn(n, n))}, T = \texttt{c*symm(randn(n, n, n))},
\]
where \(g\) is a random \(n\)-dimensional vector multiplied by a scaling factor \texttt{a}; \(H\) is a random symmetric \(n \times n\) matrix multiplied by \texttt{b}; and \(T\) is a random supersymmetric \(n \times n \times n\) tensor multiplied by \texttt{c}. Here, \(n\) represents the dimension of the problem (\(\s\in \mathbb{R}^n\)), and \texttt{symm(randn())} denotes a symmetric matrix or tensor whose entries follow a normal distribution with mean zero and variance one. The parameters \texttt{a}, \texttt{b}, and \texttt{c} are chosen differently to test the algorithm's performance under various scenarios. 
{For Algorithm \ref{alg33}, the stopping criterion is set to \(\epsilon = 10^{-5}\), specifically, \(\|\nabla_{\s} L(\s^*, \lambda^*)\| \leq 10^{-5}\). Note the stopping criterion for Step 4 is set such that the PAM (Algorithm \ref{alg32}) stops if the error reduction is less than \( \epsilon = 10^{-9} \). }
 Under this formulation, in all our examples, \(T_3\) represents nonconvex cubic polynomials.

We conducted tests for the PAM algorithm (i.e., Algorithm \ref{alg32}, with the subproblem solved using \eqref{closed form sol TR}) to minimize \(T_3\) subject to the constraint \(\|\s\| = \Delta\). The numerical results are summarized in Table \ref{table 1}. The PAM algorithm demonstrates a monotonic reduction in the function value of \(T_3\) across each iteration, as shown in Figure \ref{fig error reduction}. In all tested cases, the PAM algorithm successfully converged to the Z-eigenvalue of \eqref{gep z}. For lower-dimensional cases ($n = 2$ to $n = 6$), the minimizers obtained were verified to be the global minimizers of \(T_3(\s)\) on the constraint \(\|\s\| = \Delta\).

In our numerical examples, we tested the first- and second-order KKT conditions. 
For all the numerical examples and the identified pairs \((\s^*, \lambda^*)\), the first-order optimality conditions were approximately satisfied, with \(\|\nabla_{\s} L(\s^*, \lambda^*)\| \leq 10^{-5}\). 
Let \(S_{\perp} \in \R^{n \times (n-1)}\) represent the orthogonal complement subspace of \(\s^*\). The second-order derivative for the  Lagrangian multiplier is  $
\nabla^2_{\s\s}L(\s, \lambda) := \nabla^2_{\s\s}T_p(\s) +  \lambda I_n.
$
For all pairs \((\s^*, \lambda^*)\), we examined the minimum eigenvalue of  
\[
S_{\perp}^\top \nabla^2_{\s\s} L(\s^*, \lambda^*) S_{\perp},
\] 
and found that the Hessian is positive definite within the subspace spanned by \(S_{\perp}\). In these cases, only one eigenvalue of \(\nabla^2_{\s\s} L(\s^*, \lambda^*)\) was negative. This negative eigenvalue corresponded to the eigenvector \(\pm \s\), indicating there is only one possible descent direction—either inward or outward—from the boundary into the trust-region.

\begin{table}[!ht]
\caption{\small \textbf{Test for Nonconvex $H$.}
Parameters for the cubic polynomial and trust-region: \texttt{a} $=80$, \texttt{b} $=80$, \texttt{c} $=80$, $\Delta = 2$.  \(S_{\perp} \in \R^{n \times (n-1)}\) represents the orthogonal complement subspace of \(\s^*\). Iter. represents the number of iterations performed in Algorithm~\ref{alg32}. 
Parameters for the PAM algorithm: $\gamma_i = 8$ for $i = 1, \dotsc, 3$, $\alpha = 1$, $\s_0 = \mathbf{0} \in \R^n$.}
\centering
\begin{tabular}{ccccccc}
\toprule
\textbf{$n$} & \textbf{Iter.} & \textbf{$\lambda^*$} & \textbf{$ T_3(\s^*)$} & \textbf{$\|\nabla_{\s} L(\s^*, \lambda^*)\|$} & \textbf{$S_{\perp}^T\nabla_{\s \s}^2 L(\s^*, \lambda^*) S_{\perp}$} & \textbf{Time (s)} \\
\midrule
2  & 16   & 161.10  & -1878.00  & $3.32 \times 10^{-6}$  & $\succ 0$  & 0.0772 \\
3  & 28   & 138.70  & -965.00   & $0.75 \times 10^{-6}$  & $\succ 0$  & 0.1177 \\
4  & 41   & 160.10  & -1748.00  & $5.62 \times 10^{-6}$  & $\succ 0$  & 0.1601 \\
5  & 144  & 238.00  & -2236.00  & $8.16 \times 10^{-6}$  & $\succ 0$  & 0.5526 \\
6  & 55   & 271.00  & -2885.00  & $9.99 \times 10^{-6}$  & $\succ 0$  & 0.2104 \\
10 & 167  & 564.90  & -5777.00  & $8.52 \times 10^{-6}$  & $\succ 0$  & 0.6880 \\
20 & 195  & 760.70  & -7693.00  & $9.78 \times 10^{-6}$  & $\succ 0$  & 0.9579 \\
30 & 135  & 1015.00 & -10151.00 & $8.90 \times 10^{-6}$  & $\succ 0$  & 1.2128 \\
40 & 178  & 1096.30 & -11014.00 & $8.89 \times 10^{-6}$  & $\succ 0$  & 3.2476 \\
\bottomrule
\end{tabular}
\label{table 1}
\end{table}

\begin{figure}[!ht]
\centering
\caption{\small \textbf{Reduction in $T_3$ for Each Iteration of the PAM Algorithm:} Parameters are the same as in Table \ref{table 1}, with $n=15$ and $\Delta = 15$.}
\includegraphics[width=7cm]{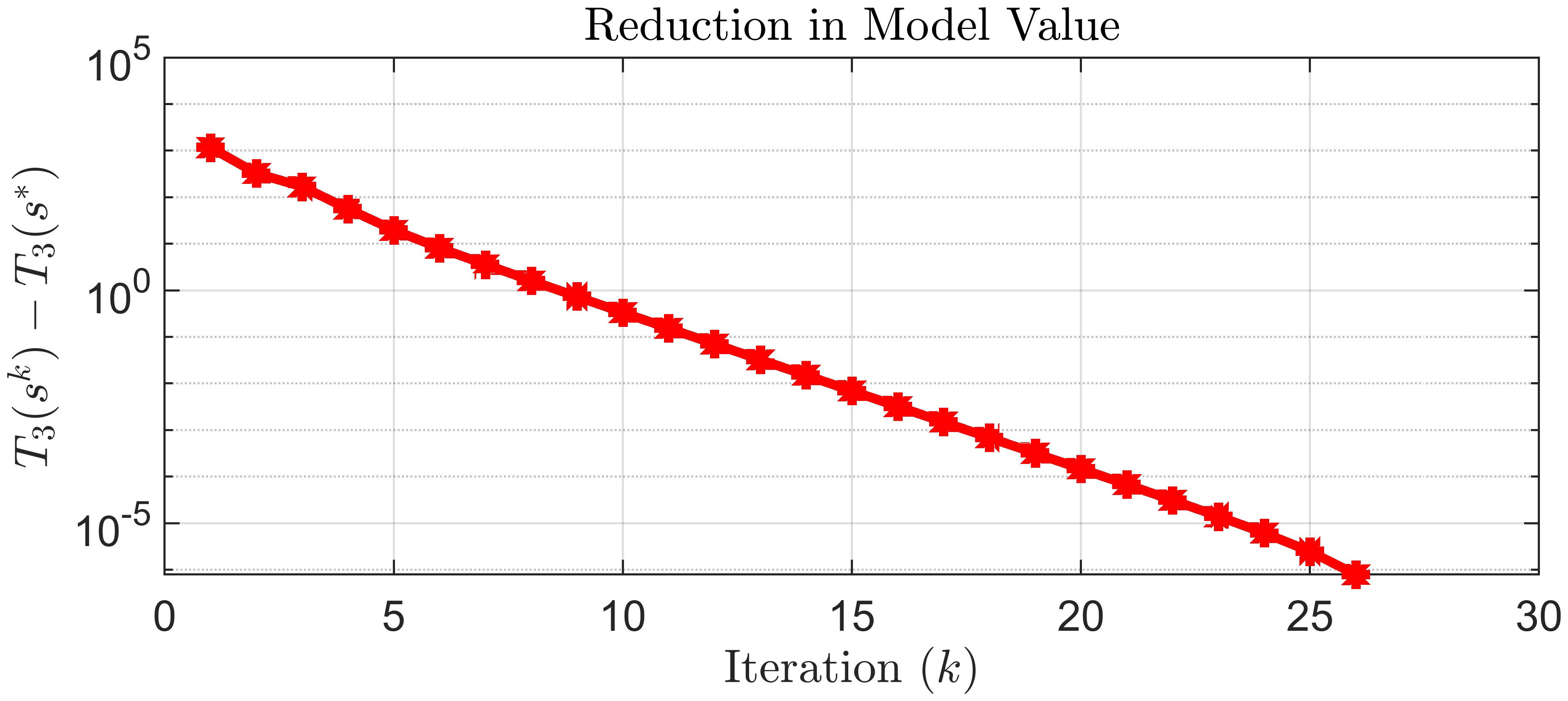}
\label{fig error reduction}
\end{figure}

We also tested the algorithm on \(T_3\) which has a small gradient $\|g\| \approx 0$ and a positive definite Hessian (\(H \succ 0\)). Such cases are considered challenging for many first- and second-order iterative minimization algorithms. When zero initialization (\(\s_0 = \mathbf{0}\)) was applied, algorithms often became trapped at \(\s_0 = \mathbf{0}\) or near a local minimum close to zero. As outlined in Table \ref{table 2}, the PAM algorithm addressed this issue by compelling the step size to move away from the local minimum at \(\s_0 = \mathbf{0}\) and converging to a lower minimum value at \(\|\s\| = \Delta\).



\begin{table}[!ht]
\caption{\small \textbf{Hard Case: Test for Near-Zero $g$ and Convex $H$.} Parameters for the cubic polynomials: \texttt{a} $=10^{-8}$, $H = \texttt{symm(40*randn(n)+ n*eye(n))}$, where $H \succ 0$, $\Delta = 2$, and \texttt{c} $=40$. Parameters for the PAM algorithm is kept the same as Table \ref{table 1}.}
\centering
\begin{tabular}{ccccccc}
\toprule
\textbf{$n$} & \textbf{Iter.} & \textbf{$\lambda^*$} & \textbf{$ T_3(\s^*)$} & \textbf{$\|\nabla_{\s} L(\s^*, \lambda^*)\|$} & \textbf{$S_{\perp}^T\nabla_{\s \s}^2 L(\s^*, \lambda^*) S_{\perp}$} & \textbf{Time (s)} \\
\midrule
2  & 17   & 148.10  & -1660.70  & $1.41 \times 10^{-6}$  & $\succ 0$  & 0.0794 \\
4  & 34   & 156.30  & -1646.70  & $5.57 \times 10^{-6}$  & $\succ 0$  & 0.1337 \\
5  & 81   & 235.10  & -2123.70  & $9.09 \times 10^{-6}$  & $\succ 0$  & 0.3150 \\
6  & 69   & 250.10  & -2540.80  & $9.71 \times 10^{-6}$  & $\succ 0$  & 0.2542 \\
10 & 165  & 540.70  & -5400.40  & $9.98 \times 10^{-6}$  & $\succ 0$  & 0.6586 \\
20 & 342  & 721.00  & -7272.40  & $9.19 \times 10^{-6}$  & $\succ 0$  & 1.7527 \\
30 & 349  & 919.00  & -9092.00  & $9.85 \times 10^{-6}$  & $\succ 0$  & 4.0148 \\
40 & 220  & 1008.70 & -9765.80  & $9.03 \times 10^{-6}$  & $\succ 0$  & 4.3313 \\
\bottomrule
\end{tabular}
\label{table 2}
\end{table}

In Figure \ref{fig delta}, we investigate the effect of varying \(\Delta\) on \(\lambda^*\) and the number of PAM iterations required to reach the minimum. Our numerical experiments consistently reveal that \(\lambda^*\) remains positive across all \(\Delta > 0\). The value of \(\lambda^*\) decreases as \(\Delta\) increases ( \(\lambda^* \to 0\) as \(\Delta \to \infty\), and \(\lambda^* \to \infty\) as \(\Delta \to 0\) ) as illustrated by the first plot of Figure \ref{fig delta}. For smaller \(\lambda^*\) values (or equivalently larger trust-region radii \(\Delta\)), \(\min_{\|\s\| = \Delta} T_3(\s^*)\) achieves a lower value as shown in the second plot of Figure \ref{fig delta}. This outcome aligns with expectations: in the trust-region algorithm, a larger trust-region radius \(\Delta\) typically corresponds to a less conservative local subproblem model, allowing for longer step sizes and greater decreases in the model value. Regarding the iteration count, there is a general tendency for the iteration count to increase as the trust-region expands (see the third plot of Figure \ref{fig delta}).

\begin{figure}[!ht]
    \centering

\caption{\small \textbf{Change of $\lambda^*$, $T_3(\s^*)$, and PAM Iter. Count for Varied $\Delta \in [1, 10]$:} The parameters are kept the same as those specified in Table \ref{table 1} with $n=15$. }
\includegraphics[width=12cm]{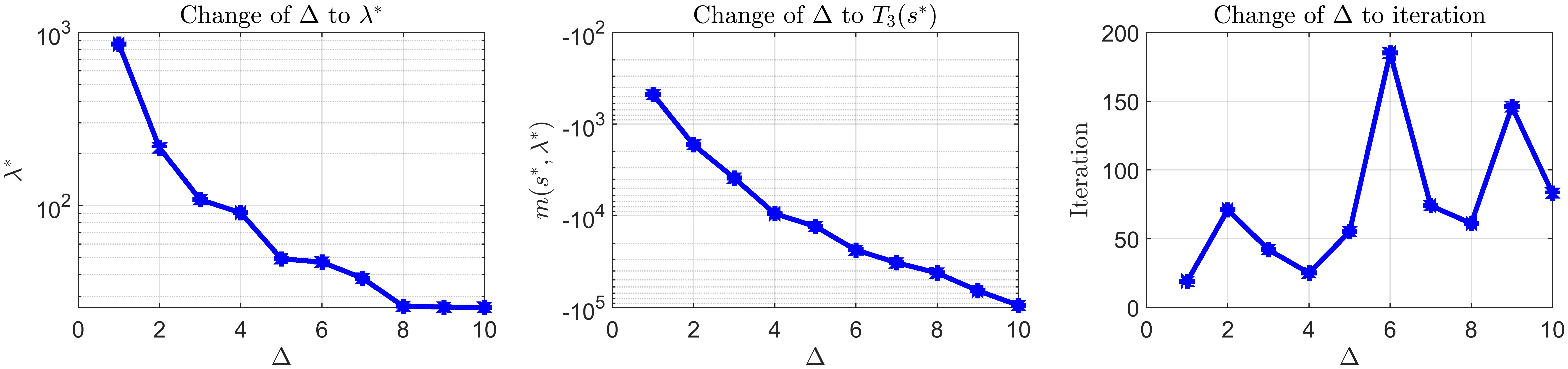}
\label{fig delta}
\end{figure}

To conclude this section, we remark that in this subsection, we have only addressed the minimizer of $T_3$ on the boundary where \(\|\s^*\| = \Delta\). For cases where the minimizer satisfies \(\|\s^*\|<\Delta\), we have $ \nabla T_p(\s) = \textbf{0} $. By choosing a symmetric homogeneous tensor $\B$ such that \(\B[\s_c]^{p} = 1\) equivalent to \(\s^\top \nabla T_p(\s) = 0\) for any $\u \in \R^n$, we can relate the first-order critical point of $T_p(\s)$ for $\|\s\|<\Delta$ with the $B$-eigenvalue problem.
\[
\min_{\s_c \in \R^{n+1}} \T[\s_c]^{p} \quad \text{s.t.} \quad \B[\s_c]^{p} = 1,
\]
where \(\B\{(0, \dotsc, 0)\} := 1\) and
\[
\B\big\{\pi(i_1, \dotsc, i_{p})\big\} := \frac{(p-|\alpha|)!\alpha_1!\cdots\alpha_n!}{(p-1)!} f_\alpha.
\]
Note that $
\s^\top \nabla T_p(\s) =  \sum_{\alpha\in\mathbb{Z}^n_+} \alpha f_{\alpha}\s^{\alpha}\in \mathbb{R}_{p}[s]$,
where
$\s^{\alpha}=s_1^{\alpha_1}s_2^{\alpha_2}\cdots s_n^{\alpha_n},
|\alpha|=\alpha_1+\alpha_2+\cdots+\alpha_n\leq p,$
$\alpha_1,\alpha_2,\cdots,\alpha_n\in\{0, 1,\dotsc,p\}$ and $f_{[0, \dotsc,0]}=0$. 
Further explorations of the PAM algorithm under the high-order trust-region model or adaptive regularization framework with efficient numerical linear techniques \cite{Xu2024UTV,Xu2024,Schnabel1984Tensor,Schnabel1991Tensor} are left for future work.

\section{Conclusions}\label{Sec5}

In this paper,  we proposed a novel tensor-based Dinkelbach--Type method for computing extremal tensor generalized eigenvalues. We established the equivalence between a homogeneous polynomial optimization problem with a spherical constraint and a multilinear programming problem with multiple spherical constraints. Leveraging this equivalence, we addressed a class of homogeneous single-ratio fractional programs and proposed an efficient proximal alternating minimization (PAM) algorithm to solve the original problem. The algorithm's subsequential and global sequence convergence were rigorously proven. We have tested our method numerically on several problems from the literature,
including computing of Z-, H-, and D-eigenpairs. Numerical experiments confirmed the efficiency and practical applicability of the proposed method.

Several intriguing and open questions remain for future investigation. First, when the objective in the polynomial optimization problem is an inhomogeneous function, does a similar equivalence with a multilinear optimization model still hold? For the homogeneous case, can the spherical constraint be generalized to other compact constraint sets? Furthermore, can the conjecture posed in Remark \ref{remark conjecture} be resolved? Although these questions present significant theoretical challenges, their resolution would yield valuable insights for implementable tensor methods.

\smallbreak
\noindent

\medskip
\begin{acknowledgements}
This work was supported by the National Natural Science Foundation of China (12071249), Shandong Provincial Natural Science Foundation for Distinguished Young Scholars (ZR2021JQ01), Shandong Provincial Natural Science Foundation (ZR2024MA003), Hong Kong Innovation and Technology Commission (InnoHK Project CIMDA).
\end{acknowledgements}

\appendix

\section{Convergence of Algorithm \ref{alg31}}
\label{appendix convergence of algo 1}

\begin{proposition}\label{prop2} {\rm\cite{CFS1985}}
Assume $\{\theta_k\}$ and $\{\x^{(k)}\}$ are infinite sequences generated by Algorithm \ref{alg31}. Then the following statements hold.

\noindent(1) For all $k$, $\theta_k\geq\theta_{k+1}\geq\bar{\theta}$.

\noindent(2) For all $k$, $F(\theta_k)\leq F(\theta_{k+1})\leq0$.

\noindent(3) The sequence $\{\theta_k\}$ converges linearly to $\bar{\theta}$, and each convergent subsequence of $\{\x^{(k)}\}$ converges to an optimal solution of (\ref{e1}).
\end{proposition}

\section{Kurdyka-{\L}ojasiewicz (KL) Property}
\label{appendix KL prop}


\begin{definition}[Proper, Lower Semicontinuous Function]
A function $\varphi: \mathbb{R}^n \to \mathbb{R} \cup \{+\infty\}$ is called a proper, lower semicontinuous function if it satisfies the following conditions:
\begin{enumerate}
    \item \textbf{Properness:} $\varphi(\mathbf{x}) < +\infty$ for at least one $\mathbf{x} \in \mathbb{R}^n$, and $\varphi(\mathbf{x}) > -\infty$ for all $\mathbf{x} \in \mathbb{R}^n$.
    \item \textbf{Lower Semicontinuity:} For any sequence $\{\mathbf{x}_k\} \subset \mathbb{R}^n$ that converges to $\mathbf{x} \in \mathbb{R}^n$, the following holds:
    \[
    \varphi(\mathbf{x}) \leq \liminf_{k \to \infty} \varphi(\mathbf{x}_k).
    \]
\end{enumerate}
\end{definition}

\begin{definition}[KL Property]
Let $\varphi: \mathbb{R}^n \to \mathbb{R} \cup \{+\infty\}$ be a proper, lower semicontinuous function. The function $\varphi$ is said to satisfy the Kurdyka-{\L}ojasiewicz (KL) property at a point $\bar{\mathbf{x}} \in \text{dom}(\partial \varphi)$ if there exist $\eta \in (0, +\infty)$, a neighborhood $U$ of $\bar{\mathbf{x}}$, and a continuous concave function $\psi: [0, \eta) \to [0, +\infty)$ such that:
\begin{enumerate}
    \item $\psi(0) = 0$,
    \item $\psi$ is differentiable on $(0, \eta)$ with $\psi'(s) > 0$ for all $s \in (0, \eta)$,
    \item For all $\mathbf{x} \in U$ such that $\varphi(\mathbf{x}) > \varphi(\bar{\mathbf{x}})$ and $\varphi(\mathbf{x}) - \varphi(\bar{\mathbf{x}}) < \eta$, the following inequality holds:
    \[
    \psi'(\varphi(\mathbf{x}) - \varphi(\bar{\mathbf{x}})) \cdot \|\partial \varphi(\mathbf{x})\| \geq 1,
    \]
    where $\partial \varphi(\mathbf{x})$ denotes the limiting subdifferential of $\varphi$ at $\mathbf{x}$.
\end{enumerate}
\end{definition}

\section{Coefficientions for Examples}
\label{appendix coeff}

\subsection{Coefficients for Example \ref{example2}}
$$
\begin{array}{llll}
a_{1111}=~~0.2883, &~a_{1112}=-0.0031, &~a_{1113}=~~0.1973, &~a_{1122}=-0.2485,\\
a_{1123}=-0.2939, &~a_{1133}=~~0.3847, &~a_{1222}=~~0.2972, &~a_{1223}=~~0.1862,\\
a_{1233}=~~0.0919, &~a_{1333}=-0.3619, &~a_{2222}=~~0.1241, &~a_{2223}=-0.3420,\\
a_{2233}=~~0.2127, &~a_{2333}=~~0.2727, &~a_{3333}=-0.3054. &
\end{array}
$$

\subsection{Coefficients for Example \ref{example3}}

$$
\begin{array}{llll}
a_{111111}=~~0.2888,&~a_{111112}=-0.0013,  &~a_{111113}=-0.1422, &~a_{111114}=-0.0323,\\
a_{111122}=-0.1079, &~a_{111123}=-0.0899,  &~a_{111124}=-0.2487, &~a_{111133}=~~0.0231,\\
a_{111134}=-0.0106, &~a_{111144}=~~0.0740, &~a_{111222}=~~0.1490,&~a_{111223}=~~0.0527,\\
a_{111224}=-0.0710, &~a_{111233}=-0.1039,  &~a_{111234}=-0.0250, &~a_{111244}=~~0.0169,\\
a_{111333}=~~0.2208,&~a_{111334}=~~0.0662, &~a_{111344}=~~0.0046,&~a_{111444}=~~0.0943,\\
a_{112222}=-0.1144, &~a_{112223}=-0.1295,  &~a_{112224}=-0.0484, &~a_{112233}=~~0.0238,\\
a_{112234}=-0.0237, &~a_{112244}=~~0.0308, &~a_{112333}=~~0.0142,&~a_{112334}=~~0.0006,\\
a_{112344}=-0.0044, &~a_{112444}=~~0.0353, &~a_{113333}=~~0.0947,&~a_{113334}=-0.0610,\\
a_{113344}=-0.0293, &~a_{113444}=~~0.0638, &~a_{114444}=~~0.2326,&~a_{122222}=-0.2574,\\
a_{122223}=~~0.1018,&~a_{122224}=~~0.0044, &~a_{122233}=~~0.0248,&~a_{122234}=~~0.0562,\\
a_{122244}=~~0.0221,&~a_{122333}=~~0.0612, &~a_{122334}=~~0.0184,&~a_{122344}=~~0.0226,\\
a_{122444}=~~0.0247,&~a_{123333}=~~0.0847, &~a_{123334}=-0.0209, &~a_{123344}=-0.0795,\\
a_{123444}=-0.0323, &~a_{124444}=-0.0819,  &~a_{133333}=~~0.5486,&~a_{133334}=-0.0311,\\
a_{133344}=-0.0592, &~a_{133444}=~~0.0386, &~a_{134444}=-0.0138, &~a_{144444}=~~0.0246,\\
a_{222222}=~~0.9207,&~a_{222223}=-0.0908,  &~a_{222224}=~~0.0633,&~a_{222233}=~~0.1116,\\
a_{222234}=-0.0318, &~a_{222244}=~~0.1629, &~a_{222333}=~~0.1797,&~a_{222334}=-0.0348,\\
a_{222344}=-0.0058, &~a_{222444}=~~0.1359, &~a_{223333}=~~0.0584,&~a_{223334}=-0.0299,\\
a_{223344}=-0.0110, &~a_{223444}=~~0.1375, &~a_{224444}=-0.1405, &~a_{233333}=~~0.3613,\\
a_{233334}=~~0.0809,&~a_{233344}=~~0.0205, &~a_{233444}=~~0.0196,&~a_{234444}=~~0.0226,\\
a_{244444}=-0.2487, &~a_{333333}=~~0.6007, &~a_{333334}=-0.0272, &~a_{333344}=-0.1343,\\
a_{333444}=-0.0233, &~a_{334444}=-0.0227,  &~a_{344444}=-0.3355, &~a_{444444}=-0.5937.
\end{array}
$$

\subsection{Coefficients for Example \ref{example4}}

$$
\begin{array}{llll}
a_{1111}=~~0.4982, &~a_{1112}=-0.0582, &~a_{1113}=-1.1719, &~a_{1122}=~~0.2236,\\
a_{1123}=-0.0171,  &~a_{1133}=~~0.4597,&~a_{1222}=~~0.4880,&~a_{1223}=~~0.1852,\\
a_{1233}=-0.4087,  &~a_{1333}=~~0.7639,&~a_{2222}=~~0.0000,&~a_{2223}=-0.6162,\\
a_{2233}=~~0.1519, &~a_{2333}=~~0.7631,&~a_{3333}=~~2.6311,
\end{array}
$$
and
$$
\begin{array}{llll}
b_{1111}=~3.0800, &~b_{1112}=~0.0614, &~b_{1113}=~0.2317, &~b_{1122}=~0.8140,\\
b_{1123}=~0.0130, &~b_{1133}=~2.3551, &~b_{1222}=~0.0486, &~b_{1223}=~0.0616,\\
b_{1233}=~0.0482, &~b_{1333}=~0.5288, &~b_{2222}=~1.9321, &~b_{2223}=~0.0236,\\
b_{2233}=~1.8563, &~b_{2333}=~0.0681, &~b_{3333}=~16.0480.&
\end{array}
$$

\section{Definition for Tensor Eigenvalues}
\label{Appendix: Def of Tensor eig}

\begin{table}[ht]
\centering
\caption{Comparison of D-, B-, H-, and Z-Eigenpairs.}
\begin{tabular}{@{}p{4cm}p{7.5cm}@{}}
\toprule
\textbf{Aspect}              & \textbf{Details} \\ \midrule
\textbf{D-Eigenpair}         & \(\mathcal{A}\x^{m-1} = \lambda \mathcal{D} \x\), $\|\x\| = 1$. Used for diagonalized structures. \\
\textbf{B-Eigenpair}         & \(\mathcal{A}\x^{m-1} = \lambda \mathcal{B}\x^{m-1}\) , $\|\B\x\| = 1$. Generalized eigenvalues for two tensors. \\
\textbf{H-Eigenpair}         & \(\mathcal{A}\x^{m-1} = \lambda \x, \|\x\| = 1, \x \geq 0\). Constrained to unit sphere and nonnegative. \\
\textbf{Z-Eigenpair}         & \(\mathcal{A}\x^{m-1} = \lambda \x, \|\x\| = 1\). Simplest symmetric form. \\ \midrule

\textbf{Constraint on \(\x\)} &
Diagonal structure for D-Eigenpair, B-dependent constraints for B-Eigenpair, nonnegativity for H-Eigenpair, and unit norm for Z-Eigenpair. \\

\textbf{Tensors Involved}    &
D-Eigenpair: \(\mathcal{A}\) and  \(\mathcal{D}\) (diagonal); B-Eigenpair: \(\mathcal{A}, \mathcal{B}\); H- and Z-Eigenpair: \(\mathcal{A}\) only. \\

\textbf{Special Properties}  &
Diagonalization (D), general optimization (B), nonnegative decomposition (H), and principal component analysis (Z). \\

\textbf{Applications}        &
D: Data decomposition, diffusion kurtosis imaging (DKI); B: Optimization tasks; H: Nonnegative matrix/tensor factorization; Z: Symmetric tensor decompositions. \\ \bottomrule
\end{tabular}
\label{tab:vertical_eigenpair_comparison}
\end{table}




\end{document}